\documentclass[10pt]{article} 
\usepackage{amsthm,amsfonts,amsmath,amssymb,latexsym,epsfig,graphics,color} 
\usepackage[matrix,arrow,curve]{xy} 
\usepackage[applemac]{inputenc}
\usepackage{mathdots}
\title{The Gysin exact sequence for $S^1$-equivariant\\ 
  symplectic homology} 
 
\author{Fr\'ed\'eric {\sc Bourgeois}, \ Alexandru \sc{Oancea} 
           \\ {\it \qquad} \\ 
         {\it \small Laboratoire de Math\'ematiques d'Orsay, UMR 8628,} \\
          {\it \small  Universit\'e Paris-Sud \& CNRS, Orsay, France} \\ 
	{\it \qquad } \\
         {\it \small Institut de Math\'ematiques de Jussieu-Paris Rive Gauche,} \\ 
         {\it \small UMR 7586, Universit\'e Pierre et Marie Curie \& CNRS , Paris, France}  
         }

\date{December 20, 2013}

\newtheorem{PARA}{}[section] 
\newtheorem{theorem}[PARA]{Theorem} 
\newtheorem{corollary}[PARA]{Corollary} 
\newtheorem{lemma}[PARA]{Lemma} 
\newtheorem{proposition}[PARA]{Proposition} 
\newtheorem{definition}[PARA]{Definition} 
 
\theoremstyle{definition} 
\newtheorem{remark}[PARA]{Remark} 
\newtheorem{example}[PARA]{Example} 
\numberwithin{equation}{section} 
\newcommand{\para}{\begin{PARA}\rm} 
\newcommand{\arap}{\end{PARA}\rm} 
\newcommand{\dfn}{\begin{definition}\rm} 
\newcommand{\nfd}{\end{definition}\rm} 
\newcommand{\rmk}{\begin{remark}\rm} 
\newcommand{\kmr}{\end{remark}\rm} 
\newcommand{\xmpl}{\begin{example}\rm} 
\newcommand{\lpmx}{\end{example}\rm} 
\newcommand{\cA}{\mathcal{A}}

\newcommand{\cC}{\mathcal{C}}

\newcommand{\cH}{\mathcal{H}} 
 
\newcommand{\cJ}{\mathcal{J}} 
 
\newcommand{\cL}{\mathcal{L}} 
\newcommand{\cM}{\mathcal{M}}

\newcommand{\cO}{\mathcal{O}} 
\newcommand{\cP}{\mathcal{P}}

\newcommand{\cW}{\mathcal{W}}

\newcommand{\og}{{\overline{\gamma}}} 
\newcommand{\ug}{{\underline{\gamma}}} 
\newcommand{\oev}{\overline{\mathrm{ev}}} 
\newcommand{\uev}{\underline{\mathrm{ev}}}

\newcommand{\olambda}{{\overline{\lambda}}} 
\newcommand{\ulambda}{{\underline{\lambda}}} 
\newcommand{\op}{{\overline{p}}} 
\newcommand{\up}{{\underline{p}}}

\newcommand{\one} 
{{{\mathchoice \mathrm{ 1\mskip-4mu l} \mathrm{ 1\mskip-4mu l} 
\mathrm{ 1\mskip-4.5mu l} \mathrm{ 1\mskip-5mu l}}}} 

\newcommand{\C}{{\mathbb{C}}}

\newcommand{\N}{{\mathbb{N}}} 
 
\newcommand{\R}{{\mathbb{R}}}

\renewcommand{\u}{{\mathbf{u}}}

\newcommand{\Z}{{\mathbb{Z}}} 
\newcommand{\ind}{\mathrm{ind}} 
\newcommand{\Jreg}{\cJ_{\mathrm{reg}}}   

\newcommand{\reg}{{\mathrm{reg}}} 
\newcommand{\eps}{{\varepsilon}} 
\newcommand{\om}{{\omega}}

\newcommand{\tf}{{\widetilde{f}}}

\def\NABLA#1{{\mathop{\nabla\kern-.5ex\lower1ex\hbox{$#1$}}}} 
\def\Nabla#1{\nabla\kern-.5ex{}_{#1}} 
\def\Tabla#1{\Tilde\nabla\kern-.5ex{}_{#1}} 
\renewcommand{\Tilde}{\widetilde}

\newcommand{\p}{{\partial}} 
\newcommand{\dbar}{{\bar\partial}}

\begin{document} 
 
\maketitle 
 
 
\begin{abstract} 
We define $S^1$-equivariant symplectic homology for symplectically aspherical manifolds with contact boundary, 
using a Floer-type construction first proposed by Viterbo. We show that it is related to the usual symplectic homology by a Gysin exact sequence. As an important ingredient of the proof, we define a parametrized version of symplectic homology, corresponding to families of Hamiltonian functions indexed by a finite dimensional smooth parameter space. 
\end{abstract} 
 
\tableofcontents 
 
 
\section{Introduction} 
 
The purpose of the current paper is twofold. On the one hand we provide the details of the definition of $S^1$-equivariant symplectic homology following Viterbo~\cite{V} and using the analysis developed in~\cite{BOtrans,BOparam}. On the other hand we construct the Gysin sequence and prove that it is compatible with the tautological exact sequence for symplectic homology. Along the way we are led to define a parametrized version of Floer homology which serves as an interpolating device between non-equivariant and equivariant theories. The purpose of the current paper is foundational, but we also include some simple applications. 

Our paper~\cite{BO4} gives a broader treatment of $S^1$-equivariant symplectic homology and contains alternative -- and more algebraic -- proofs of Theorems~\ref{thm:SGysin} and~\ref{thm:grid}. However, we believe that the geometric methods that we use in the current paper are interesting on their own.

\medskip
 
\noindent {\sc Topological background.} Given an oriented fibration 
$S^1\hookrightarrow M \stackrel \pi \to B$, the homology groups of the 
base and total space are related by the {\bf Gysin exact sequence}  
\begin{equation} \label{eq:Gysin} 
\ldots \to H_k(M) \stackrel {\pi_*} \to H_k(B) \stackrel D \to H_{k-2}(B) 
\to H_{k-1}(M) \to \ldots 
\end{equation} 
Here $D$ is the cap-product with the Euler class of the fibration and 
is equal to the differential $d^2$ of the Leray-Serre spectral 
sequence~\cite[Example~5.C]{McC}.  
 
A particular case of the above construction is the following. Assume $M$ 
carries an $S^1$-action and define the $S^1$-equivariant homology 
$H_*^{S^1}(M)$ by  
$$ 
H_*^{S^1}(M) := H_*(M_{S^1}), \qquad M_{S^1}:= M\times_{S^1} ES^1,   
$$ 
where $ES^1$ is a contractible space on which $S^1$ acts freely.  
Since $S^1$ acts freely on $M\times ES^1$, the 
projection $M \times ES^1 \to M_{S^1}$ is an $S^1$-fibration and the 
exact sequence~\eqref{eq:Gysin} becomes 
\begin{equation} \label{eq:equiGysin} 
 \ldots \to H_k(M) \to H_k^{S^1}(M) \stackrel D \to H_{k-2}^{S^1}(M)  
\to H_{k-1}(M) \to \ldots 
\end{equation} 
We call this the {\bf Gysin exact sequence for \boldmath$S^1$-equivariant 
homology}.  
Two relevant instances of this construction are the following: 
 
(i) If the action of $S^1$ on $M$ is free 
then $H_*^{S^1}(M)\simeq H_*(M/S^1)$ and the Gysin exact sequence for 
$S^1$-equivariant homology is the Gysin exact sequence for the fibration 
$S^1\hookrightarrow M \to M/S^1$.  
 
(ii) We denote $BS^1:=ES^1/S^1$. Taking the model of $ES^1$ to be 
$S^\infty := \lim_{N\to \infty} S^{2N+1}$, with $S^{2N+1}$ the unit 
sphere in $\C^{N+1}$, we see that $BS^1\simeq \C P^\infty$. Now, if 
$S^1$ acts trivially on $M$, then   
$H_*^{S^1}(M)\simeq H_*(M)\otimes H_*(BS^1)$ and~\eqref{eq:equiGysin} 
becomes    
$$ 
\ldots \stackrel 0 \to H_k(M) \stackrel i \to \bigoplus _{m\ge 0} 
H_{k-2m}(M) \stackrel p \to \bigoplus _{m\ge 1} H_{k-2m}(M)  
\stackrel 0 \to H_{k-1}(M) \to \ldots 
$$ 
Here $i$ and $p$ are the obvious inclusion and projection. 
 
\medskip  
 
\noindent {\sc Main results.} This paper is concerned with a Floer homology long exact sequence 
of Gysin type. Let $(W,\om)$ be a symplectic manifold with contact 
type boundary satisfying 
\begin{equation} \label{eq:asph} 
\int _{T^2} f^*\om =0 \quad \mbox{for all smooth } f:T^2\to W. 
\end{equation}  
Our main class of examples consists of exact symplectic manifolds. Let 
$a$ be a free homotopy class of loops in $W$. One can define in this 
situation {\bf symplectic homology groups} $SH_*^a(W)$ and {\bf 
\boldmath$S^1$-equivariant symplectic homology groups} $SH_*^{a,S^1}(W)$, as 
well as variants $SH_*^+(W)$, $SH_*^{+,S^1}(W)$ truncated in positive 
values of the action functional when $a=0$. The original definition was outlined by Viterbo~\cite{V} and we present all the details in~\S\ref{sec:S1equivsymplhom}. Our first result is the following.  
 
\begin{theorem} \label{thm:SGysin} 
The symplectic homology groups fit into an exact sequence of 
Gysin type (we allow $a=+$) 
\begin{equation} \label{eq:SGysin} 
 \ldots \to SH_k^a(W) \to SH_k^{a,S^1}(W) \stackrel D \to 
SH_{k-2}^{a,S^1}(W) \to SH_{k-1}^a(W) \to \ldots 
\end{equation} 
\end{theorem} 
 
As a matter of fact, we prove in~\cite{BO4} that the above Gysin exact sequence for $a=+$ is isomorphic to the long exact sequence of~\cite{BOcont}, relating $SH_*^+(W)$ with the linearized contact homology of the filled contact manifold $\p W$. 
 
In the case $a=0$, the symplectic 
homology groups   
$$ 
SH_*(W):= SH_*^0(W), \qquad SH_*^{S^1}(W):= SH_*^{0,S^1}(W) 
$$ 
also fit into tautological long exact sequences~\cite{V} 
\begin{equation} \label{eq:taut} 
\ldots \to SH_{*+1}^+(W) \to H_{*+n}(W,\p W) \to 
SH_*(W) \to SH_*^+(W) \to \ldots \ , 
\end{equation} 
\begin{equation} \label{eq:tautS1} 
\ldots \to SH_{*+1}^{+,S^1}(W) \to H_{*+n}^{S^1}(W,\p W) \to 
SH_*^{S^1}(W) \to SH_*^{+,S^1}(W) \to \ldots  
\end{equation} 
Here the $S^1$-equivariant homology of the pair $(W,\p W)$ is 
considered with respect to the trivial action of $S^1$. Our next 
result is that the Gysin exact sequence is compatible with these 
tautological exact sequences.  
 
\begin{theorem} \label{thm:grid}  
There is a commutative diagram whose rows and columns are, respectively, the tautological and Gysin exact sequences  
\begin{equation} \label{eq:grid}  
\xymatrix 
@C=10pt 
@R=6pt 
{ 
& \vdots \ar[d] & \vdots \ar[d] & \vdots \ar[d] & \vdots \ar[d] & \\  
\ldots \ar[r] & SH_{k+1}^+ \ar[r] \ar[d] & H_{k+n} \ar[r] \ar[d] & 
SH_k \ar[r] \ar[d] & SH_k^+ \ar[r] \ar[d] & \ldots \\ 
\ldots \ar[r] & SH_{k+1}^{+,S^1} \ar[r] \ar[d] & H_{k+n}^{S^1} \ar[r] \ar[d] & 
SH_k^{S^1} \ar[r] \ar[d] & SH_k^{+,S^1} \ar[r] \ar[d] & \ldots \\ 
\ldots \ar[r] & SH_{k-1}^{+,S^1} \ar[r] \ar[d] & H_{k+n-2}^{S^1} 
\ar[r] \ar[d] & SH_{k-2}^{S^1} \ar[r] \ar[d] & SH_{k-2}^{+,S^1} \ar[r] 
\ar[d] & \ldots \\  
\ldots \ar[r] & SH_k^+ \ar[r] \ar[d] & H_{k+n-1} \ar[r] \ar[d] & 
SH_{k-1} \ar[r] \ar[d] & SH_{k-1}^+ \ar[r] \ar[d] & \ldots \\ 
& \vdots & \vdots & \vdots & \vdots &  
} 
\end{equation}  
\end{theorem}

\medskip 

\noindent {\sc Examples.} We discuss the consequences of our main theorems for two important classes of symplectic manifolds. 

\smallskip \noindent {\it Cotangent bundles.} Let $L$ be a closed oriented and spin Riemannian manifold, and denote by $\Lambda L$ the free loop space of $L$. We consider  the symplectic manifold $W=DT^*L=\{p\in T^*L\, : \, \|p\|\le 1\}$. It was proved by Viterbo~\cite{Vcotangent} that 
$$
SH_*(DT^*L)\simeq H_*(\Lambda L), \qquad SH_*^{S^1}(DT^*L)\simeq H_*^{S^1}(\Lambda L).
$$ 
Alternative proofs for the first isomorphism are due to Abbondandolo and Schwarz~\cite{AS1}, respectively to Salamon and Weber~\cite{SW}. Moreover, the homology groups $H_*(\Lambda L)$ and $H_*(\Lambda L)$ should be understood as twisted by local coefficients given by the second Stiefel-Whitney class of $L$ (Kragh, Seidel, Abouzaid~\cite{AS3}).
Our proof of Theorem~\ref{thm:SGysin} can be combined with the methods of~\cite{AS1} in order to prove that the long exact sequence~\eqref{eq:SGysin} is isomorphic to the Gysin sequence for $\Lambda L$, namely
\begin{equation} \label{eq:Gysinloop}
\xymatrix
@C=20pt
{
\dots \ar[r] & H_*(\Lambda L) \ar[r]^-E & H_*^{S^1}(\Lambda L) \ar[r]^-D &
H_{*-2}^{S^1}(\Lambda L) \ar[r]^-M & H_{*-1}(\Lambda L) \ar[r] & \dots
}
\end{equation}
Similarly, for $a=+$, we obtain the Gysin sequence of the pair $(\Lambda^0L,L)$, where $\Lambda^0L$ is the component of free contractible loops in $L$.

\smallskip \noindent {\it Subcritical Stein manifolds.} A subcritical Stein manifold is a complex manifold $(W,J)$, of complex dimension $n$, endowed with a pluri-subharmonic function $\phi:W\to \R$, satisfying the following conditions: (i) the boundary $\p W$ is a regular level set of $\phi$ along which $\vec\nabla \phi$ points outwards; (ii) $\phi$ is Morse and the index of all its critical points is strictly smaller than $n$. The complex structure $J$ is compatible with the natural symplectic form $\omega_\phi:=-d(d\phi\circ J)$. 

It was proved by Cieliebak~\cite{C} that $SH_*(W)=0$. His proof can be adapted in a straightforward way in order to show that $SH_*^{S^1}(W)=0$. However, this fact follows also from Theorem~\ref{thm:SGysin} in the case $c_1(W)=0$. 

\begin{corollary} \label{cor:subcrit}
Assume $W$ is a subcritical Stein manifold with $c_1(W)=0$. Then we have $SH_*^{S^1}(W)=0$ and there is an isomorphism of exact sequences
$$
{\scriptsize
\xymatrix
@C=10pt
{
\dots \ar[r] & SH_*^+(W) \ar[r] \ar[d]_\simeq & SH_*^{+,S^1}(W) \ar[r]^-D \ar[d]_\simeq &
SH_{*-2}^{+,S^1}(W) \ar[r] \ar[d]_\simeq & SH_{*-1}^+(W) \ar[r] \ar[d]_\simeq & \dots \\
\dots \ar[r]^-0 & H_{*+n-1}(W,\p W) \ar[r] & H_{*+n-1}^{S^1}(W,\p W) \ar[r] & H_{*+n-3}^{S^1}(W,\p W) \ar[r]^-0 & H_{*+n-2}(W,\p W) \ar[r] & \dots
}
}
$$
\end{corollary}

\begin{proof}
Applying Theorem~\ref{thm:SGysin} we obtain that $SH_k^{S^1}(W)\simeq SH_{k-2}^{S^1}(W)$ for all $k\in \Z$.
It was proved by M.-L.~Yau~\cite[Theorem~3.1.III, Lemma~4.2]{Y} that one can choose the plurisubharmonic function $\phi$ so that the Conley-Zehnder indices of all closed characteristics on $\p W$ are positive. (Note that the Conley-Zehnder indices are well-defined due to the assumption $c_1=0$.) It follows from the definition of $S^1$-equivariant symplectic homology in~\S\ref{sec:S1equivsymplhom} that the underlying chain complex is zero if the degree is small enough (one can use "split" Hamiltonians as in the proof of Lemma~\ref{lem:minus}). Reasoning by induction, it follows that $SH_*^{S^1}(W)=0$. The isomorphism of exact sequences follows immediately from Theorem~\ref{thm:grid}, since the columns involving $SH_*$ and $SH_*^{S^1}$ vanish identically. 
\end{proof}

\medskip 

\noindent {\sc Algebraic Weinstein conjecture.} Following Viterbo~\cite{V}, we say that $W$ satisfies the \emph{Strong Algebraic Weinstein Conjecture (SAWC)} if the map 
$$
H_{2n}(W,\p W)\to SH_n(W)
$$ 
vanishes. Let $\mu_{2n}\in H_{2n}(W,\p W)$ be the fundamental class and $u_k$ be a generator of $H_{2k}(BS^1)$, $k\ge 0$.
We say that $W$ satisfies the \emph{Strong Equivariant Algebraic Weinstein Conjecture (EWC)} if, for all $k\ge 0$, the element 
$\mu_{2n}\otimes u_k$ lies in the kernel of the map 
$$
H_{2n+2k}^{S^1}(W,\p W)\to SH_{n+2k}^{S^1}(W).
$$ 
Our next result clarifies the relationship between $SAWC$ and $EWC$, which are the two key notions in Viterbo's fundamental paper~\cite{V}.
\begin{corollary} \label{cor:WC} $SAWC\Longrightarrow EWC$. 
\end{corollary}

\begin{proof}
We first note that $SAWC$ is equivalent to the vanishing of $SH_*(W)$. This follows from the fact that $SH_*(W)$ is a ring with unit~\cite{McLean}, and the unit is the image of the fundamental class $\mu_{2n}$ under the map $H_{2n}(W,\p W)\to SH_n(W)$~\cite{Seidel-biased}. 

We now consider the top middle square in the commutative diagram~\eqref{eq:grid} of Theorem~\ref{thm:grid}. Since $\mu_{2n}\otimes u_0$ is the image of $\mu_{2n}$ under the injection $H_{2n}\to H_{2n}^{S^1}$, it follows that $\mu_{2n}\otimes u_0$ is in the kernel of $H_{2n}^{S^1}\to SH_n^{S^1}$. We now prove by induction that $\mu_{2n}\otimes u_k$ is in the kernel of $H_{2n+2k}^{S^1}\to SH_{n+2k}^{S^1}$. This follows from the middle square in the commutative diagram~\eqref{eq:grid}, using that $\mu_{2n}\otimes u_{k+1}$ is sent to $\mu_{2n}\otimes u_k$ by the map $H_{2n+2k+2}^{S^1}\to H_{2n+2k}^{S^1}$, and the fact that $SH_{n+2k+2}^{S^1}\to SH_{n+2k}^{S^1}$ is an isomorphism. 
\end{proof}

\begin{remark}
 The same argument as above shows that, under the assumption $SAWC$, the maps $H_{k+n}^{S^1}\to SH_k^{S^1}$ vanish for all $k\in\Z$.
\end{remark}

\medskip

\noindent {\sc Ramifications.} We now present several directions of investigation which are related to the present paper. 

\smallskip

\noindent {\it Algebraic structures.} The Gysin exact sequence~\eqref{eq:SGysin} can be used to define algebraic operations
in ($S^1$-equivariant) symplectic homology. 

As already mentioned in the proof of Corollary~\ref{cor:WC}, symplectic homology $SH_*(W)$ is a unitary ring, with the
pair-of-pants product. This is described by
Seidel~\cite{Seidel-biased}, and was used in a crucial way by McLean~\cite{McLean}
in his construction of exotic affine $\R^{2n}$'s. We denote the pair of pants
product by 
$$
\bullet:SH_k(W)\otimes SH_\ell(W)\longrightarrow SH_{k+\ell-n}(W).
$$

Let us write the Gysin exact sequence~\eqref{eq:SGysin} as 
$$
\xymatrix
@C=20pt
{
\dots \ar[r] & SH_*(W) \ar[r]^E & SH_*^{S^1}(W) \ar[r]^D &
SH_{*-2}^{S^1}(W) \ar[r]^M & SH_{*-1}(W) \ar[r] & \dots
}
$$
The notation is motivated by the isomorphism with the exact sequence~\eqref{eq:Gysinloop} in the case $W=DT^*L$. 
The letters $M$ and $E$ stand for ``mark'' and ``erase'', in the terminology of
Chas and Sullivan~\cite{CS}. It was proved by Abbondandolo and Schwarz~\cite{AS2} that, in the case $W=DT^*L$, the pair-of-pants product 
is identified with the Chas-Sullivan loop product~\cite{CS}.

Inspired by Chas and Sullivan~\cite{CS}, we formulate the following definitions and claims, which we will prove in a forthcoming paper. 
\begin{itemize}
\item[---] The map 
$$
\Delta:SH_*(W)\to SH_{*+1}(W), \qquad \Delta:=M\circ E
$$ 
is a \emph{Batalin-Vilkovisky (BV) operator}, in the sense that 
$\Delta^2=0$, and 
$$
\{\cdot,\cdot\} : SH_k(W)\otimes SH_\ell(W)\to SH_{k+\ell-n+1}(W),
$$
$$
\{a,b\} \  := \  \pm \, \Delta(a\bullet b) \pm  a \bullet \Delta(b) \pm
b\bullet \Delta(a)
$$
is a bracket on $SH_*(W)$ (called \emph{the loop bracket}). 
\item[---] The map 
$$
[\cdot,\cdot]:SH_k^{S^1}(W)\otimes SH_\ell^{S^1}(W)\to SH_{k+\ell-n+2}^{S^1}(W)
$$
$$
[a,b]:=\pm\, E(M(a)\bullet M(b))
$$
is a bracket on $SH_*^{S^1}(W)$ (called \emph{the string bracket}).
\end{itemize} 
We give a chain-level description of $\Delta$ in Remark~\ref{rmk:Delta}. The above claims are analogous to Theorems~4.7, 5.4, and 6.1 of~\cite{CS}. The string bracket can be further generalized as follows. Any operation 
$$
\widetilde \sigma:SH_*^{\otimes k}\to SH_*, k\ge 2
$$ 
yields an operation 
$$
\sigma:=E\circ \widetilde \sigma \circ M^{\otimes k}:(SH_*^{S^1})^{\otimes k}\to SH_*^{S^1}.
$$ 
One particular case is $\widetilde \sigma:=\bullet ^{\otimes k-1}$, $k\ge 2$, which yields higher-order operations on $SH_*^{S^1}$ analogous to the ones of~\cite[Theorem~6.2]{CS}.

The range of applications of such operations depends on their explicit
knowledge in particular situations (e.g. cotangent
bundles). However, the Chas-Sullivan string
operations are only beginning to be 
understood by topologists (see the work of Felix, Thomas, and
Vigu\'e-Poirrier~\cite{Felix-Thomas-Vigue,Felix-Thomas-Vigue-2}).  
 
It should also be possible to describe these operations directly in terms of
holomorphic curves. Such a construction is sketched by Seidel in~\cite{Seidel-biased}. 

\smallskip

\noindent {\it Relation to Hochschild and cyclic homology.} 
Paul Seidel has conjectured in~\cite{Se} that, 
given an exact Lefschetz fibration $E\to D$ over the disc, the 
symplectic homology of $E$ is isomorphic to the Hochschild homology of 
a certain $A_\infty$-category $\cC$ built from the vanishing cycles of 
$E$: 
$$ 
SH_*(E) \simeq HH_*(\cC). 
$$ 
This conjecture has been proved by Ganatra and Maydanskiy in~\cite{Gan-May} as a consequence of the Legendrian handle attaching exact triangle of Bourgeois, Ekholm and Eliashberg~\cite{BEE}. It is implicit in~\cite{Se} that there is an equivariant version of 
this conjectural isomorphism, namely that the $S^1$-equivariant symplectic 
homology of $E$ is isomorphic to the cyclic homology of $\cC$: 
$$ 
SH_*^{S^1}(E)\simeq HC_*(\cC).  
$$ 
On the other hand, Hochschild and cyclic homology are related by the 
Connes exact sequence  
\begin{equation} \label{eq:Connes} 
 \ldots \to HH_k(\cC) \to HC_k(\cC) \stackrel D \to HC_{k-2}(\cC) \to 
HH_{k-1}(\cC) \to \ldots  
\end{equation} 
We conjecture that the two previous isomorphisms are such that the 
Gysin exact sequence~\eqref{eq:SGysin} and the Connes exact 
sequence~\eqref{eq:Connes} are isomorphic. This fits with the general 
philosophy that the Gysin exact sequence for $S^1$-equivariant 
homology of certain topological spaces is isomorphic to 
the Connes exact sequence of suitable algebras (a good reference is 
Loday's book~\cite{Lo}, in particular~\cite[Theorem~7.2.3]{Lo}).  

\smallskip

\noindent {\it Relation to Givental's point of view.} Given a closed symplectic manifold $X$, Givental defined in~\cite{Gi} a $D$-module structure on 
$H^*(X;\C)\otimes \Lambda_{Nov}\otimes \C[\hbar]$, where $\Lambda_{Nov}$ is a suitable Novikov ring and $\hbar$ is the generator of $H^*(BS^1)$. He interprets this 
as being the $S^1$-equivariant Floer cohomology of $X$. Our construction of $S^1$-equivariant Floer homology provides an interpretation of the underlying homology group as the homology of a Floer-type complex. We expect that the $D$-module structure can also be defined within our setup. 

\medskip
 
 \noindent {\sc Structure of the paper.} In~\S\ref{sec:symplhom} we briefly recall the construction of symplectic homology. In~\S\ref{sec:param} 
 we introduce a new variant of it, which we call ``parametrized symplectic homology''. It corresponds to families of Hamiltonians, indexed by a finite dimensional parameter space. 
 Section~\ref{sec:S1} is devoted to the $S^1$-equivariant theory. We recall in~\S\ref{sec:S1equivhom} the Borel construction and its interpretation in Morse homology. We define $S^1$-equivariant symplectic homology in~\S\ref{sec:S1equivsymplhom}, following Viterbo~\cite[\S5]{V}. We prove Theorems~\ref{thm:SGysin} and~\ref{thm:grid} in~\S\ref{sec:MBparam}, using a Morse-Bott construction and a spectral sequence argument. In~\S\ref{sec:continuation} we use similar techniques to study continuation maps. 
 
 \medskip 
 
 \noindent {\sc Acknowledgements.} F.B. was partially supported by the Fonds National de la Recherche Scientifique (Belgium) and by ERC Starting Grant StG-239781-ContactMath. A.O. was partially supported by ERC Starting Grant StG-259118-Stein. Both authors were partially
supported by ANR project ``Floer Power'' ANR-08-BLAN-0291-03 (France) as well as by the Minist\`ere Belge
des Affaires \'etrang\`eres and the Minist\`ere
Fran\c{c}ais des Affaires \'etrang\`eres et europ\'eennes through the
programme PHC--Tournesol Fran\c{c}ais. The present work is part of the authors activities within CAST, a Research ?Network Program of the European Science Foundation.


\section{Symplectic homology} \label{sec:symplhom} 
 
We briefly recall in this section the definition of symplectic 
homology, and we refer to~\cite{BOauto} for full details. In the sequel 
$(W,\om)$ denotes a compact symplectic manifold with  
contact type boundary $M:=\p W$. This means that there exists a vector 
field $X$ defined in a neighbourhood of $M$, transverse and pointing 
outwards along $M$, and such that 
$$ 
\cL _X \om = \om. 
$$ 
Such an $X$ is called a {\bf Liouville vector field}. The $1$-form 
$\alpha:=(\iota_X\om)|_M$ is a contact form on $M$ and is called the 
{\bf Liouville \boldmath$1$-form}. We denote by 
$\xi:=\ker \alpha$ the contact structure defined by $\alpha$, and 
we note that the isotopy class of $\xi$ is uniquely determined by 
$\om$. The {\bf Reeb vector field} $R_\alpha$ is defined by 
the conditions $\ker \, \om|_M = \langle R_\alpha \rangle$ and 
$\alpha(R_\alpha)=1$. We denote by $\phi_\alpha$ the flow of 
$R_\alpha$. The {\bf action spectrum} of $(M,\alpha)$ is 
defined by 
$$ 
\textrm{Spec}(M,\alpha) := \{ T \in \R^+\, | \, \textrm{ there is a 
   closed } R_\alpha\textrm{-orbit of period } T\}. 
$$ 
 
Let $\phi$ be the flow of $X$. We parametrize a neighbourhood $U$ of $M$ by 
$$ 
G: M \times [-\delta, 0] \to U, \qquad  
(y,t) \mapsto \phi^t(y). 
$$ 
Then $d(e^t\alpha)$ is a symplectic form on $M\times \R^+$ and 
$G$ satisfies $G^*\om = d(e^t \alpha)$. 
We denote by 
$$ 
\widehat W : = W \ \bigcup _{G} \ M\times \R^+ 
$$ 
the {\bf symplectic completion of \boldmath$W$} and endow it with the 
symplectic form  
$$ 
\widehat \om : = 
\left\{\begin{array}{ll} 
\om & \textrm{ on } W, \\ 
d(e^t \alpha) & \textrm{ on } M\times \R^+. 
\end{array} \right. 
$$ 
 
Given a time-dependent Hamiltonian $H :S^1\times \widehat W \to \R$ 
we define the 
{\bf Hamiltonian vector field} $X^\theta_H$ by 
$$ 
\widehat \om (X^\theta_H,\cdot) = d H_\theta, \qquad \theta\in S^1 = \R/\Z, 
$$ 
where $H_\theta:=H(\theta,\cdot)$. We denote by $\phi_H$ the flow of 
$X_H^\theta$, defined by $\phi_H^0=\textrm{Id}$ and 
$$ 
  \frac d {d\theta} \phi_H^\theta (x) = X^\theta_H(\phi_H^\theta(x)),  
\qquad \theta\in \R. 
$$ 
We denote by $\cP(H)$ the set of $1$-periodic orbits of $X^\theta_H$, 
and we denote by $\cP^a(H)\subset \cP(H)$ the set of $1$-periodic 
orbits in the free homotopy class $a$.  
 
We define the class $\cH$ of {\bf admissible Hamiltonians} to consist of 
smooth functions $H:S^1\times \widehat W\to \R$ satisfying the following 
conditions: 
\begin{itemize}  
\item $H<0$ on $W$; 
\item 
there exists $t_0\ge 0$ such that $H(\theta,y,t)=\beta e^t 
  +\beta'$ for $t\ge t_0$, with $0<\beta\notin 
  \mathrm{Spec}(M,\alpha)$ and $\beta'\in\R$. 
\end{itemize}  
 
We denote by $\cH_{\textrm{reg}}\subset \cH$ the dense set of Hamiltonians $H$ 
  such that all elements of $\cP(H)$ are nondegenerate, i.e. the Poincar\'e 
  return map has no eigenvalues equal to $1$.  
Let $a$ be a free homotopy class of loops in $W$. 
The {\bf symplectic homology groups} of $(W,\om)$ are defined by 
\begin{equation*} 
   SH_*^a(W,\om) := \lim _{\substack{ \longrightarrow \\ H\in \cH_{\textrm{reg}}} 
    } SH_*^a(H,J). 
\end{equation*} 
Here $J$ is an almost complex structure on $\widehat W$ which is 
compatible with $\widehat \om$, convex and invariant under translation 
in the $t$-variable outside a compact set, and regular for 
$H$ (in particular one must allow $J$ to depend on $\theta$).  
We denote by $SH_*^a(H,J)$ the Floer homology groups of the pair 
$(H,J)$ in the free homotopy class $a$ and with coefficients in the 
Novikov ring $\Lambda_\om$. We assume throughout this paper that $W$ 
satisfies condition~\eqref{eq:asph}, so that the energy of a Floer  
trajectory does not depend on its homology class, but only on its  
endpoints. We refer to \cite{BOauto} for the details of the construction 
and in particular for the definition of the coefficient ring 
$\Lambda_\om$. Throughout this paper the Novikov ring is understood to 
be defined over $\Z$.  
 
For the trivial homotopy class $a=0$ we denote 
the symplectic homology groups by $SH_*(W,\om)$. The {\bf 
reduced Hamiltonian action functional} is  
$$ 
\cA_H^0 : C^\infty_{\textrm{contr}}(S^1,\widehat W) \to \R, 
$$ 
$$ 
\cA_H^0(\gamma) := -\int_{D^2} \sigma^*\widehat \om - \int_{S^1} 
H(\theta,\gamma(\theta)) \, d\theta. 
$$ 
Here $C^\infty_{\textrm{contr}}(S^1,\widehat W)$ denotes the space of 
smooth contractible loops in $\widehat W$ and $\sigma:D^2\to \widehat 
W$ is a smooth extension of $\gamma$. Note that $\cA_H^0$ is 
well-defined thanks to condition~(\ref{eq:asph}) and  
is decreasing along Floer trajectories.  
 
We now consider a special cofinal class of Hamiltonians $\cH'\subset 
\cH$, consisting of elements $H\in \cH'$ which satisfy the following 
conditions:  
\begin{itemize} 
\item 
there exists $t_0\ge 0$ such that $H(\theta,y,t)=\beta e^t 
  +\beta'$ for $t\ge t_0$, with $0<\beta\notin 
  \mathrm{Spec}(M,\alpha)$ and $\beta'\in\R$; 
\item $H<0$ and $C^2$-small on $W$; 
\item 
$H(\theta,y,t)$ is $C^2$-close to an increasing function of $t$ 
  on $S^1\times M \times [0,t_0]$. 
\end{itemize} 
The last condition implies that, in the region $M\times [0,t_0]$, each 
$1$-periodic orbit of $H$ is $C^1$-close to a closed characteristic on 
some level $M\times\{t\}$. 
 
Given $H\in\cH'_\reg:=\cH_\reg\cap \cH'$, a regular almost complex 
structure $J$, and a choice of $\epsilon>0$ small enough, we define 
the chain complexes   
\begin{equation} \label{eq:SC-} 
SC_*^-(H,J) := \bigoplus _{\substack{ \gamma \in 
     \cP^0(H) \\ \cA_H^0(\gamma) \le \epsilon }} \Lambda_\om 
\langle \gamma \rangle \ \subset SC_*(H,J) 
\end{equation} 
and 
$$ 
SC_*^+(H,J) := SC_*(H,J) / SC_*^-(H,J). 
$$ 
The differential on $SC_*^\pm(H,J)$ is induced by $\p$. The groups 
$$ 
SH_*^\pm(H,J) := H_*(SC_*^\pm(H),\p) 
$$ 
do not depend on $J$, nor on $\epsilon$, and we define 
$$ 
SH_*^\pm(W,\om) := \lim _{\substack{ \to \\ H\in \cH'_{\textrm{reg}}}} SH_*^\pm(H). 
$$ 
We call $SH_*^+(W,\om)$ the {\bf positive symplectic homology group} 
of $(W,\om)$. 
   
\begin{remark} {\rm 
Condition~\eqref{eq:asph} can be  
replaced in the case of contractible orbits by the weaker {\bf 
symplectic asphericity} condition $\langle \om,\pi_2(W)\rangle =0$.  
} 
\end{remark} 
 
Let us assume now that $W$ has {\bf positive contact 
   type} boundary~\cite[\S5.4]{O1}. This means that every positively oriented 
closed characteristic $\gamma$ on $M$ which is contractible in $W$ has 
positive action $\cA_\om(\gamma)$ bounded away from zero, where 
$$ 
\cA_\om(\gamma) := \int_{D^2} \sigma^*\om 
$$ 
for some extension $\sigma:D^2\to W$ of $\gamma$. This condition is 
automatically satisfied if the boundary $M$ is of restricted contact 
type, i.e. the vector field $X$ is globally defined on $W$. Under the 
positive contact type assumption we have~\cite{V} 
$$ 
SH_*^-(W,\om) = H_{*+n}(W,\p W;\Lambda_\om), \qquad n=\frac 1 2 \dim 
\, W, 
$$ 
and the short exact sequence of complexes $SC_*^-(H) \to SC_*(H) 
\to SC_*^+(H)$ induces the tautological long exact 
sequence~\eqref{eq:taut}.


\section{Parametrized symplectic homology} \label{sec:param} 
 
We introduce in this section a new variant of Floer 
homology, which we call ``parametrized Floer homology''.  
In the sequel $\Lambda$ is a finite dimensional closed 
manifold of dimension $m$, which we call ``parameter space''. The 
elements of $\Lambda$ are denoted by $\lambda$. When the 
parameter space is $S^{2N+1}$, the parametrized symplectic homology 
groups will be the abutment of the spectral sequence which gives rise 
to the Gysin exact sequence~\eqref{eq:SGysin}.

\subsection{The parametrized Floer equation} \label{sec:action} 
For each free homotopy class $a$ in $W$, we fix a reference loop 
$l_a:S^1\to \widehat W$ such that $[l_a]=a$. If $a$ is the trivial 
homotopy class, we choose $l_a$ to be a constant loop.  
Recall that free homotopy classes of loops in $\widehat W$ are in 
one-to-one correspondence with conjugacy classes in $\pi_1(\widehat 
W)$. As a consequence, the inverse $a^{-1}$ of a free homotopy class 
is well-defined. We require that $l_{a^{-1}}$ coincides with the loop 
$l_a$ with the opposite orientation. 
 
We define the set $\cH_\Lambda$ of 
{\bf admissible Hamiltonian families} to consist of elements  
$H\in C^\infty(S^1\times \widehat W\times \Lambda,\R)$ which satisfy 
the following conditions: 
\begin{itemize}  
\item $H<0$ on $S^1\times W\times \Lambda$;  
\item 
there exists $t_0\ge 0$ such that $H(\theta,y,t,\lambda)=\beta 
  e^t +\beta'(\lambda)$ for $t\ge t_0$, with $0<\beta\notin 
  \mathrm{Spec}(M,\alpha)$ and $\beta'\in C^\infty(\Lambda,\R)$.
\end{itemize}  
 
Let $H:S^1 \times \widehat W \times \Lambda \to 
\R$ be an admissible Hamiltonian family denoted by 
$H(\theta,x,\lambda)=H_\lambda(\theta,x)$. This defines a family of  
action functionals  
$$ 
\cA : C^\infty(S^1,\widehat W)\times \Lambda \to \R, 
$$ 
$$ 
\cA(\gamma,\lambda) = \cA_\lambda(\gamma) := -\int_{[0,1]\times S^1} \sigma^* 
\om - \int_{S^1} H_\lambda(\theta,\gamma(\theta)) d\theta,  
$$ 
where $\sigma:[0,1]\times S^1\to \widehat W$ is a smooth homotopy from 
$l_{[\gamma]}$ to $\gamma$. The functional $\cA$ is well-defined due 
to our standing assumption~\eqref{eq:asph}.  
 
The differential of $\cA$ is given by  
\begin{equation} \label{eq:dA} 
d\cA(\gamma,\lambda) \cdot (\zeta,\ell)= 
\int_{S^1}\om(\dot\gamma(\theta)-X_{H_\lambda}(\gamma(\theta)),\zeta(\theta)) 
d\theta 
- 
\int_{S^1} \frac {\p H} {\p \lambda} (\theta,\gamma(\theta),\lambda) d\theta 
\cdot \ell 
\end{equation}  
and therefore $(\gamma,\lambda)$ is a critical point of $\cA$ if and 
only if  
\begin{equation} \label{eq:periodicpar} 
 \gamma\in\cP(H_\lambda) \quad \mbox{and} \quad  
 \int_{S^1} \frac {\p H} {\p \lambda} (\theta,\gamma(\theta),\lambda)\, 
d\theta =0.  
\end{equation}  
We denote by $\cP(H)$ the set of critical points of $\cA$ consisting 
of pairs $(\gamma,\lambda)$ satisfying~\eqref{eq:periodicpar}. We 
denote by $\cP^a(H)$ the set of pairs $(\gamma,\lambda)\in\cP(H)$ such 
that $\gamma$ lies in the free homotopy class $a$.  
 
\begin{remark} {\rm  
 Equation~\eqref{eq:periodicpar} can be interpreted as follows. Every 
loop $\gamma:S^1\to\widehat W$ determines a function  
\begin{equation} \label{eq:Fgamma} 
F_\gamma:\Lambda \to \R, \qquad \lambda \mapsto \int_{S^1} 
H(\theta,\gamma(\theta),\lambda) \, d\theta.  
\end{equation}  
A pair $(\gamma,\lambda)$ belongs therefore to $\cP(H)$ if and only if 
$$ 
\gamma\in \cP(H_\lambda) \quad \mbox{ and } \quad \lambda\in 
\textrm{Crit}(F_\gamma). 
$$  
} 
\end{remark}  
 
Let $J=(J_\lambda^\theta)$, $\lambda\in\Lambda$, $\theta\in S^1$ be a 
family of $\theta$-dependent compatible  
almost complex structures on $\widehat W$ 
which, at infinity, are invariant under translations in the 
$t$-variable and satisfy the relations  
\begin{equation} \label{eq:standardJ} 
J_\lambda^\theta \xi=\xi, \qquad J_\lambda^\theta (\frac \partial 
{\partial t}) =R_\alpha. 
\end{equation} 
Such an {\bf admissible family of almost complex 
  structures} $J$ induces a  
  family of $L^2$-metrics on the space $C^\infty(S^1,\widehat W)$, parametrized 
  by $\Lambda$ and defined by   
$$ 
\langle \zeta,\eta\rangle_\lambda := \int_{S^1} 
\om(\zeta(\theta),J_\lambda^\theta\eta(\theta)) d\theta, \quad \zeta,\eta\in 
T_\gamma C^\infty(S^1,\widehat 
W)=\Gamma(\gamma^*T\widehat W). 
$$ 
Such a metric can be coupled with any metric $g$ on $\Lambda$ and 
gives rise to a metric on  
$C^\infty(S^1,\widehat W)\times \Lambda$ acting  
at a point $(\gamma,\lambda)$ by  
$$ 
\langle(\zeta,\ell), (\eta,k)\rangle_{J,g}:= \langle \zeta,\eta\rangle_\lambda + 
g(\ell,k), \qquad (\zeta,\ell),(\eta,k)\in \Gamma(\gamma^*T\widehat 
W)\oplus T_\lambda\Lambda.  
$$ 
We denote by $\cJ_\Lambda$ the set of pairs $(J,g)$ consisting of an 
admissible almost complex structure $J$ on $\widehat W$ and of a 
Riemannian metric $g$ on $\Lambda$.  
 
The {\bf parametrized Floer equation} is the gradient equation for 
$\cA$ with respect to such a metric 
$\langle\cdot,\cdot\rangle_{J,g}$. More precisely, given   
$\op:=(\og,\olambda),\up:=(\ug,\ulambda)\in \cP(H)$  
we denote by  
$$ 
\widehat \cM(\op,\up;H,J,g) 
$$ 
the {\bf space of parametrized Floer trajectories}, consisting of 
pairs $(u,\lambda)$ with  
$$ 
u:\R\times S^1 \to \widehat W, \qquad \lambda:\R\to \Lambda,  
$$ 
satisfying  
\begin{eqnarray}  
\label{eq:Floer1par} 
 \p_s u + J_{\lambda(s)}^\theta (\p_\theta u - 
X_{H_{\lambda(s)}}^\theta (u)) & = & 0, \\  
\label{eq:Floer2par} 
 \dot \lambda (s) - \int_{S^1} \vec \nabla_\lambda 
H(\theta,u(s,\theta),\lambda(s)) d\theta & = & 0,   
\end{eqnarray} 
and  
\begin{equation} \label{eq:asymptoticpar} 
 \lim_{s\to -\infty} (u(s,\cdot),\lambda(s)) = (\og,\olambda), \quad  
 \lim_{s\to +\infty} (u(s,\cdot),\lambda(s)) = (\ug,\ulambda). 
\end{equation} 
Here and in the sequel we use the notation $\vec \nabla$ for a 
gradient vector field, whereas $\nabla$ will denote a 
covariant derivative.  
 
\begin{remark}{\rm Equation~\eqref{eq:Floer2par} is equivalent to  
\begin{equation} \label{eq:Floer2parbis} 
 \dot \lambda(s) - \vec \nabla F_{u(s,\cdot)}(\lambda(s))=0, 
\end{equation}  
where $F_{u(s,\cdot)}$ is defined by~\eqref{eq:Fgamma}. Thus, the 
parametrized Floer equation is a system involving a Floer equation and 
a finite-dimensional gradient equation.  
} 
\end{remark}  
 
The additive group $\R$ acts on $\widehat \cM(\op,\up;H,J,g)$ 
by reparametrization in the $s$-variable and we denote by  
$$ 
\cM(\op,\up;H,J,g) := \widehat \cM(\op,\up;H,J,g)/\R 
$$ 
the {\bf moduli space of parametrized Floer trajectories}.

Let us fix $p\ge 2$. The linearization of the 
equations~(\ref{eq:Floer1par}-\ref{eq:Floer2par}) gives rise to the  
operator  
\begin{equation} \label{eq:Dulambda}
D_{(u,\lambda)} : W^{1,p}(u^*T\widehat W) \oplus 
W^{1,p}(\lambda^*T\Lambda) \to  
L^p(u^*T\widehat W) \oplus L^p(\lambda^*T\Lambda), 
\end{equation}
$$ 
D_{(u,\lambda)} (\zeta,\ell) :=  
\left(\begin{array}{c}  
D_u\zeta + (D_\lambda J\cdot \ell)(\p_\theta u - X_{H_\lambda}(u)) - 
J_\lambda (D_\lambda X_{H_\lambda}\cdot \ell) \\ 
\nabla_s \ell - \nabla_\ell \int_{S^1} \vec \nabla_\lambda H 
(\theta,u,\lambda) d\theta  
- \int_{S^1} \nabla_\zeta \vec \nabla_\lambda H(\theta,u,\lambda) d\theta 
\end{array}\right), 
$$ 
where 
$$ 
D_u : W^{1,p}(u^*T\widehat W) \to L^p(u^*T\widehat W) 
$$ 
is the usual Floer operator given by  
$$ 
D_u\zeta := \nabla_s \zeta + J_\lambda \nabla_\theta \zeta - 
J_\lambda \nabla_\zeta X_{H_\lambda} + \nabla_\zeta J_\lambda (\p_\theta u - 
X_{H_\lambda}). 
$$ 
 
The Hessian of $\cA$ at a critical point $p=(\gamma,\lambda)$ is given 
by the formula  
\begin{eqnarray} \label{eq:d2A} 
\lefteqn{d^2\cA(\gamma,\lambda)\big((\zeta,\ell),(\eta,k)\big)} \\ 
&=& \int_{S^1} \omega(\nabla_\theta\eta-\nabla_\eta 
X_{H_\lambda},\zeta) d\theta - \int_{S^1}\eta(\frac {\partial H} 
{\partial \lambda}\cdot \ell) d\theta \nonumber \\ 
&& - \ \int_{S^1} k(dH_\lambda\cdot \zeta) d\theta - \int_{S^1} \frac 
{\partial^2 H} {\partial \lambda^2}(\ell,k) d\theta \nonumber \\ 
&=& d^2\cA_{H_\lambda}(\gamma)(\zeta,\eta) - \int_{S^1}\eta(\frac {\partial H} 
{\partial \lambda}\cdot \ell) d\theta  
- \int_{S^1} k(dH_\lambda\cdot \zeta) d\theta  
- d^2 F_\gamma(\lambda)(\ell,k). \nonumber 
\end{eqnarray}  
 
We define the asymptotic operator at a critical point 
$(\gamma,\lambda)$ by 
$$ 
D_{(\gamma,\lambda)} : H^1(S^1,\gamma^*T\widehat W) \times T_\lambda 
\Lambda \to L^2(S^1,\gamma^*T\widehat W) \times T_\lambda 
\Lambda, 
$$ 
\begin{equation}  \label{eq:Dasy} 
D_{(\gamma,\lambda)}(\zeta,\ell) = \left(\begin{array}{c}  
J_\lambda(\nabla_\theta \zeta - \nabla_\zeta X_{H_\lambda} - 
(D_\lambda X_{H_\lambda})\cdot \ell) \\  
-\int_{S^1} \nabla_\zeta \frac {\partial H} {\partial \lambda} d\theta 
- \int_{S^1} \nabla_\ell \frac {\partial H} {\partial \lambda} d\theta 
\end{array}\right). 
\end{equation} 
Note that $D_{(\gamma,\lambda)}$ is obtained from $D_{(u,\lambda)}$ 
for $(u(s,\theta),\lambda(s))\equiv (\gamma(\theta),\lambda)$ and 
$(\zeta(s,\theta),\ell(s)) \equiv (\zeta(\theta),\ell)$.  
 
We say that a critical point $(\gamma,\lambda)$ is {\bf nondegenerate} if 
the Hessian $d^2\cA(\gamma,\lambda)$ has trivial kernel. In~\cite[Lemma~2.3]{BOtrans} we proved that this condition is equivalent to the injectivity of the asymptotic  
operator $D_{(\gamma,\lambda)}$. Since the latter is self-adjoint, this condition is also equivalent to 
its surjectivity.  
 
\begin{remark} 
  We note that nondegeneracy of a critical point $(\gamma,\lambda)$ 
  does not imply that $\gamma$ is a nondegenerate orbit of 
  $H_\lambda$, nor that $\lambda$ is a nondegenerate critical point of 
  $F_\gamma$. This situation is already present in Morse theory, as 
  the following example shows. We consider the Morse function 
  $f:\R\times \R\to\R$, $(x,y)\mapsto xy$. Then $(x_0,y_0)=(0,0)$ is a 
  nondegenerate critical point, but the restrictions of $f$ to 
  $\R\times \{0\}$ and $\{0\}\times \R$ are constant, hence $x_0=0$ 
  and $y_0=0$ are degenerate critical points.  
\end{remark}

An admissible Hamiltonian family $H$ is called {\bf nondegenerate} if $\cP(H)$ 
consists of nondegenerate elements.   
We denote the set of nondegenerate and admissible Hamiltonian families by 
$\cH_{\Lambda,\textrm{reg}}\subset \cH_\Lambda$. By~\cite[Proposition~2.4]{BOtrans}, the 
set $\cH_{\Lambda,\textrm{reg}}$ is of the second Baire category in 
$\cH_\Lambda$. Moreover, if $H\in\cH_{\Lambda,\textrm{reg}}$ the set 
$\cP(H)$ is discrete.   
 
We denote 
\begin{eqnarray*} 
\cW^{1,p} & := & W^{1,p}(\R\times S^1,u^*T\widehat W) \oplus 
W^{1,p}(\R,\lambda^*T\Lambda), \\  
\cL^p & := & L^p(\R\times S^1,u^*T\widehat W) \oplus 
L^p(\R,\lambda^*T\Lambda).
\end{eqnarray*} 
Let $(\og,\olambda),(\ug,\ulambda)\in\cP(H)$ be nondegenerate. We proved in~\cite[Theorem~2.5]{BOtrans} that, given any $(u,\lambda)\in 
\widehat \cM((\og,\olambda),(\ug,\ulambda); H,J,g)$, the operator 
$$ 
D_{(u,\lambda)} : \cW^{1,p}\to \cL^p 
$$  
is Fredholm for $1<p<\infty$.

\begin{remark} We can choose a unitary trivialization of 
$u^*T\widehat W$  
and a trivialization of $\lambda^*T\Lambda$ in which $D_{(u,\lambda)}$ has 
the form   
\begin{equation} \label{eq:Dtriv} 
D_{(u,\lambda)}\left(\begin{array}{c} 
\zeta \\ \ell  
\end{array}\right) 
:= 
\Bigg[\left(\begin{array}{cc} 
\partial_s +J_0\partial_\theta & 0 \\ 0 & \frac d {ds}  
\end{array}\right)  
+ N 
\Bigg] 
\left(\begin{array}{c} 
\zeta \\ \ell  
\end{array}\right),  
\end{equation} 
with $N:\R\times S^1\to \mathrm{Mat}_{2n+m}(\R)$ being pointwise 
bounded and $\lim_{s\to\pm\infty}N(s,\theta)$ being
symmetric. To obtain such a trivialization we just need to pick a unitary trivialization of $u^*T\widehat W$. Pointwise boundedness then follows from the fact that the trajectory converges at $\pm\infty$, whereas symmetry follows from the fact that the asymptotic operator at a critical point is self-adjoint.
\end{remark}

Let $H\in\cH_{\Lambda,\mathrm{reg}}$. A pair $(J,g)\in\cJ_\Lambda$ is 
  called {\bf regular for \boldmath$H$} if the operator $D_{(u,\lambda)}$ is 
  surjective for any solution $(u,\lambda)$ 
  of~(\ref{eq:Floer1par}-\ref{eq:asymptoticpar}). We denote the space 
  of such pairs by $\cJ_{\Lambda,\mathrm{reg}}(H)$. We proved in~\cite[Theorem~4.1]{BOtrans} that there exists a subset of second Baire category
$\cH\cJ_{\Lambda,\mathrm{reg}}\subset 
\cH_{\Lambda,\mathrm{reg}}\times \cJ_\Lambda$ such that $H\in\cH_{\Lambda,\mathrm{reg}}$ and $(J,g)\in
\cJ_{\Lambda,\mathrm{reg}}(H)$ whenever $(H,J,g)\in
\cH\cJ_{\Lambda,\mathrm{reg}}$. 

As a consequence, whenever $(H,J,g)\in\cH\cJ_{\Lambda,\mathrm{reg}}$ we infer that the moduli spaces of parametrized Floer trajectories 
$\cM(\op,\up;H,J,g)$ are smooth manifolds, for all $\op,\up\in\cP(H)$. The local dimension at $(u,\lambda)\in \cM(\op,\up;H,J,g)$ is equal to $\mathrm{ind}\,D_{(u,\lambda)}-1$. 

Recall that, for each free homotopy class $a$ in $\widehat W$, we have chosen 
in Section~\ref{sec:action} a reference loop $l_a$ such that 
$[l_a]=a$. We now choose a symplectic trivialization 
$$ 
\Phi^1_a:S^1\times \R^{2n} \to l_a^*T\widehat W 
$$ 
for each free homotopy class $a$. If $a$ is the trivial homotopy 
class we choose the trivialization to be constant.   
 
For each $p=(\gamma,\lambda)\in\cP(H)$ we choose a smooth 
homotopy $\sigma_p:[0,1]\times S^1\to\widehat W$ such that 
$\sigma_p(0,\cdot)=l_{[\gamma]}$ and $\sigma_p(1,\cdot)=\gamma$.  
This gives rise to a unique (up to homotopy) symplectic trivialization 
$$ 
\Phi^1_p: [0,1]\times S^1\times \R^{2n} \to \sigma^*_pT\widehat W 
$$ 
such that $\Phi^1_p=\Phi^1_{[\gamma]}$ on $\{0\} \times S^1 \times 
\R^{2n}$. Moreover, we fix an isometry $\Phi^2_p:\R^m\to 
T_\lambda\Lambda$.   

Let $T^*\Lambda$ be the cotangent bundle, denote points in $T^*\Lambda$ by $(\lambda,\eta)$ with $\eta\in T^*_\lambda\Lambda$, and endow $T^*\Lambda$ with the symplectic form $d\lambda\wedge d\eta$. Following~\cite{BOparam} we define a Hamiltonian $\widetilde H:S^1\times \widehat W\times T^*\Lambda\to \R$ by 
$\widetilde H(\theta,x,(\lambda,\eta)):=H(\theta,x,\lambda)$. Then $X_{\widetilde H}=X_H - \frac {\p H} {\p \lambda} \frac \p {\p \eta}$ and periodic orbits of $X_{\widetilde H}$ project onto critical points of $\cA$ in $\widehat W \times \Lambda$. 
Together, $\Phi^1_p$ and $\Phi^2_p$ induce a symplectic trivialization
of $T(\widehat W \times T^*\Lambda)$ along the closed Hamiltonian orbits 
of $\widetilde H$ in $\widehat W \times T^*\Lambda$ which project to $p$.

In~\cite{BOparam}, we defined the {\bf parametrized 
  Robbin-Salamon index} $\mu(p)$ of $p$ with respect to the given 
trivialization as the Robbin-Salamon index~\cite{RS} of the linearized 
Hamiltonian flow of $\widetilde H$ along a closed orbit over $p$.
The main result from~\cite{BOparam} can be phrased in terms of our parametrized
 Floer equations as:
 
\begin{theorem} \label{thm:index}  
Assume $(\og,\olambda),(\ug,\ulambda)\in\cP(H)$ are 
  nondegenerate and fix $1<p<\infty$. For any $(u,\lambda)\in 
\widehat \cM((\og,\olambda),(\ug,\ulambda); H,J,g)$ the index of the 
Fredholm operator $D_{(u,\lambda)}:\cW^{1,p}\to\cL^p$ is  
$$ 
\ind \, D_{(u,\lambda)} = -\mu(\og, \olambda) + \mu(\ug, \ulambda) .  
$$ 
\end{theorem}  
 
In the above statement, it is understood that the trivialization used 
to define $\mu(\og,\olambda)$ is obtained from the trivialization used 
to define $\mu(\ug,\ulambda)$ by continuation along the map $u$. 

\subsection{The parametrized chain complex} 
Given $H\in\cH_{\Lambda,\mathrm{reg}}$, $(J,g)\in\Jreg(H)$, and a free 
homotopy class $a$ in $\widehat W$, we 
define $SC^{a,\Lambda}_*(H,J,g)$ as a chain complex whose underlying 
$\Lambda_\omega$-module is  
$$ 
SC^{a,\Lambda}_*(H,J,g):=\bigoplus _{p\in\cP^a(H)} 
\Lambda_\omega\langle p \rangle. 
$$ 
We define the degree of a generator $p\in\cP(H)$ in terms of the 
parametrized Robbin-Salamon index by 
$$ 
|p|:= -\mu(p)+\frac m 2\in \Z. 
$$ 
The fact that the grading is integral follows 
from~\cite[Theorem~4.7]{RS} using that the periodic orbits of $X_{\widetilde H}$ form Morse-Bott families of dimension $m=\dim\,\Lambda$, or from the integrality property stated in~\cite[Appendix~B]{BOparam}. 
We define $|p\, e^A|:=|p|-2\langle  
c_1(T\widehat W),A\rangle$, where $c_1(T\widehat W)$ is computed with 
respect to a compatible almost complex structure. 
 
Recall that, for each $p=(\gamma,\lambda)\in\cP(H)$, we have chosen a 
cylinder $\sigma_p:[0,1]\times S^1\to\widehat W$ such that 
$\sigma_p(0,\cdot)=l_{[\gamma]}$ and $\sigma_p(1,\cdot)=\gamma$. We 
define $\overline \sigma_p(s,\theta):=\sigma_p(1-s,\theta)$. Given 
$\op=(\og,\olambda),\up=(\ug,\ulambda)\in\cP(H)$ we define  
$$ 
\cM^A(\op,\up;H,J,g)\subset \cM(\op,\up;H,J,g) 
$$ 
to consist of trajectories $(u,\lambda)$ such that 
$[\sigma_{\op}\#u\#\overline \sigma_{\up}]=A\in H_2(\widehat 
W;\Z)$. It follows from Theorem~\ref{thm:index} that  
$$ 
\dim\, \cM^A(\op,\up;H,J,g) = |\op| - |\up\, e^A| -1. 
$$ 
 
Let $\op:=(\og,\olambda),\up:=(\ug,\ulambda)\in\cP(H)$. Whenever  
$|\op|-|\up \, e^A|=1$, one can associate to each 
element $(u,\lambda)\in \cM^A(\op,\up;H,J,g)$ a sign $\eps(u,\lambda)$ 
via the coherent orientations recipe of Floer and Hofer~\cite{FH}. As 
in their construction, since the asymptotics are fixed, the relevant 
spaces of Fredholm operators are contractible, and the corresponding 
determinant line bundles are trivial. Hence the moduli spaces of 
parametrized Floer trajectories are orientable. Since our moduli 
spaces are modeled on $\R$ as gradient trajectories, we can use the 
algorithm in~\cite{FH} to construct a set of orientations which is 
coherent with respect to the gluing operation. More precisely, one 
chooses an element $p\in\cP(H)$, and for each $p\neq \up\in\cP(H)$ one 
chooses arbitrary orientations of the spaces of operators 
$\cO(p,\up)$ asymptotic to $D_p$ at $-\infty$ and to $D_{\up}$ at 
$+\infty$. These determine orientations of $\cO(\up,p)$ by requiring 
that the glued orientation on $\cO(p,p)$ be the one determined by the 
canonical orientation of the constant operator $D_p$. We obtain 
orientations on $\cO(\op,\up)$ by requiring that the glued orientation 
with $\cO(p,\op)$ and $\cO(\up,p)$ be the canonical one on 
$\cO(p,p)$.  
 
We define a differential $\partial$ on $SC^{a,\Lambda}_*(H,J,g)$ by 
\begin{equation} \label{eq:diff-param}
\partial\op:=\sum_{|\op|-|\up \, e^A|=1} \Big( \sum_{(u,\lambda)\in 
  \cM^A(\op,\up;H,J,g)} \eps(u,\lambda)\Big) \ \up \, e^A. 
\end{equation} 
This expression is well-defined by standard compactness 
arguments~\cite{HS,S}. More precisely, for each $A\in H_2(W;\Z)$ 
satisfying $|\op|-|\up\, e^A|=1$ the 
set $\cM^A(\op,\up;H,J,g)$ is finite, and for each $c>0$ the number of 
$A\in H_2(W;\Z)$ such that $\omega(A)\le c$ and 
$\cM^A(\op,\up;H,J,g)\neq\emptyset$ is finite.  
 
It follows from standard compactness and gluing arguments~\cite{F,S} that 
$\partial^2=0$. Compactness is established in three steps. Firstly, one 
obtains a uniform $C^0$-bound on the $\widehat W$-component  
of parametrized Floer trajectories using the
maximum principle~\cite[Lemma~1.5]{O} 
and the fact that outside a compact set $H$ is independent of $\theta\in S^1$ and $y\in M$ and satisfies 
$\partial^2 H(t,\lambda(s))/\partial s \partial t=0$. 
Secondly, one proves 
that the $\Lambda$-component converges by applying the 
Arzel\'a-Ascoli theorem. Thirdly, the $\widehat W$-component converges 
by Floer-Gromov compactness because it satisfies an $s$-dependent 
Floer equation. Gluing involves exactly the same kind of estimates as 
in Floer theory.  
 
We denote the resulting homology groups by 
$SH^{a,\Lambda}_*(H,J,g)$. As for usual symplectic homology, we obtain by 
passing to the direct limit {\bf parametrized symplectic homology groups}  
$$ 
SH^{a,\Lambda}_*(W) := \lim _{\stackrel \longrightarrow 
  {H\in\cH_{\Lambda,\mathrm{reg}}}} SH^{a,\Lambda}_*(H,J,g).  
$$ 
 
\begin{proposition}[K\"unneth formula] \label{prop:Kunneth} 
 The following isomorphism holds with field coefficients 
\begin{equation}  
 SH^{a,\Lambda}_*(W) \simeq SH^a_*(W) \otimes H_*(\Lambda).  
\end{equation}  
\end{proposition}  
 
\begin{proof} We use Hamiltonians of the form  
$$ 
H_\lambda(\theta,x):= K(\theta,x) + f(\lambda).  
$$ 
Here $f:\Lambda\to \R$ is a Morse function and $K$ is an admissible 
Hamiltonian having nondegenerate orbits. We choose a generic 
admissible almost complex structure $J$ on $W$ and a generic 
Riemannian metric $g$ on $\Lambda$.  
 
The critical points of the parametrized action functional are of the 
form $(\gamma,\lambda)$, $\gamma\in \cP(K)$, 
$\lambda\in \mathrm{Crit}(f)$. The  
properties of the parametrized Robbin-Salamon index described 
in~\cite[Proposition 4]{BOparam} imply that  
$$ 
\mu(\gamma,\lambda)=\mu_{RS}(\gamma) + \mathrm{ind}_f(\lambda)-\frac m 
2, 
$$ 
where $\mathrm{ind}_f(\lambda)$ denotes the Morse index of 
$\lambda\in\mathrm{Crit}(f)$. It follows that  
$$ 
|(\gamma,\lambda)|=-\mu_{RS}(\gamma) + m - \mathrm{ind}_f(\lambda) = 
-\mu_{RS}(\gamma) + \mathrm{ind}_{-f}(\lambda). 
$$ 
The parametrized Floer equation is 
split and has the form   
$$ 
\left\{\begin{array}{rcl} \dbar_J u & = & J X_{H_\lambda} = JX_K, \\ 
\dot \lambda(s) & = & \int_{S^1} \vec \nabla_\lambda H(\theta, 
u(s,\theta), \lambda(s)) d\theta = \vec \nabla f(\lambda(s)). 
\end{array}\right.  
$$  
This follows from the obvious identities  
$$ 
X_{H_\lambda}(\theta,x,\lambda) \equiv X_K(\theta,x), 
\qquad 
\vec \nabla_\lambda H(\theta,x,\lambda) \equiv \vec \nabla 
f(\lambda). 
$$ 
We obtain an isomorphism of complexes  
$$ 
SC^{a,\Lambda}_*(H,J,g) \simeq SC^a_*(K,J) \otimes C_*(-f,g), 
$$ 
where $SC^a_*(K,J)$ denotes the Floer complex for $(K,J)$ in the free 
homotopy class $a$ (graded by 
$-\mu_{RS}(\gamma)$) and $C_*(-f,g)$ denotes the Morse complex for 
$(-f,g)$ (graded by $\mathrm{ind}_{-f}(\lambda)$). Since we use 
field coefficients the conclusion follows by the algebraic K\"unneth 
theorem.   
\end{proof}  
 
\begin{remark}{\bf (Naturality)} \label{rmk:SHincl} 
 An embedding of parameter spaces $\iota:\Lambda \hookrightarrow \Lambda'$ induces 
 a natural map $S\iota_*:SH^{a,\Lambda}_*(W)\to SH^{a,\Lambda'}_*(W)$ which is equal 
 to $\mathrm{Id}\otimes \iota_*$ via the K\"unneth isomorphism. This 
 can be seen by using a Hamiltonian 
 $K(\theta,x)+f(\lambda)$ on $S^1\times \widehat 
 W\times \Lambda$ as in the proof of Proposition~\ref{prop:Kunneth} above, 
 and a Hamiltonian $K(\theta,x)+\widetilde f(\lambda')$ on $S^1\times \widehat 
 W\times \Lambda'$, where 
 $\widetilde f=f-|z|^2$ in a tubular 
 neighbourhood of $\Lambda\subset \Lambda'$ and $z$ is the normal 
 coordinate.  
\end{remark}


\section{\boldmath$S^1$-equivariant theories} \label{sec:S1}
 
In \S\ref{sec:S1equivhom} we  
give a Morse theoretic presentation of $S^1$-equivariant homology of a 
manifold carrying an $S^1$-action. This serves as a motivation for 
\S\ref{sec:S1equivsymplhom}, where we give the definition of the 
$S^1$-equivariant symplectic homology groups $SH_*^{S^1}(W)$ 
following Viterbo~\cite[\S5]{V}.   
We adopt a slightly more general setting and define groups  
$SH_*^{a,S^1}(W)$ corresponding to nontrivial free homotopy classes of 
loops in $W$. 
 
\subsection[$S^1$-equivariant homology and Morse theory]{\boldmath$S^1$-equivariant homology and Morse theory} 
  \label{sec:S1equivhom} 
 
In this section $M$ denotes a finite-dimensional smooth manifold 
carrying a smooth action of $S^1$. Our aim is to give a description  
of  
$$ 
H_*^{S^1}(M):=H_*(M\times_{S^1} ES^1) 
$$ 
in terms of Morse homology. We recall that  
$\displaystyle ES^1=\lim_{\stackrel \longrightarrow N} S^{2N+1}$ and 
therefore $\displaystyle M\times_{S^1} ES^1 = \lim _{\stackrel 
  \longrightarrow N} \,  
M\times_{S^1} S^{2N+1}$. We denote   
$$ 
M_{S^1} := M\times _{S^1} ES^1, \quad M_{S^1}^{(N)} := M\times _{S^1} 
S^{2N+1}. 
$$ 
   
The first observation is  
that, given a positive integer $k$, the homology groups 
$H_*(M_{S^1}^{(N)})$ stabilize in degree $*\le k$ for $N$ large 
enough. Indeed, the equivariant inclusion $S^{2N+1}\hookrightarrow 
S^{2N+3}$ induces an inclusion of fibrations  
$$ 
\xymatrix 
@R=15pt 
{M\ \ar@{^(->}[r] & M_{S^1}^{(N)} \ar[r] \ar[d] & \C P^N \ar[d] \\ 
M\ \ar@{^(->}[r] & M_{S^1}^{(N+1)} \ar[r]  & \C P^{N+1} 
} 
$$ 
This induces in turn a morphism between the associated Leray-Serre 
spectral sequences which is an isomorphism on the $E^2$-page in total 
degree less than $N$. Functoriality of the Leray-Serre spectral 
sequence implies that, for $N$ large enough (and determined by $k$), 
the above inclusion induces isomorphisms $H_*(M_{S^1}^{(N)}) \stackrel 
\sim \to H_*(M_{S^1}^{(N+1)})$, $*\le k$ (see for 
example~\cite[Theorem~3.5]{McC}).    
 
We can give a description of $H_*(M_{S^1}^{(N)})$ in terms of  
Morse-Bott functions on $M\times S^{2N+1}$ as follows.  
We choose a function $a :M \times S^{2N+1} \to \R$ which is 
$S^1$-invariant, i.e.    
$$ 
a(\tau x,\tau\lambda)=a(x,\lambda), \quad \tau \in S^1, \ 
(x,\lambda)\in M\times S^{2N+1}, 
$$ 
and which has only Morse-Bott circles of critical 
points, i.e. the induced function $\underline a: 
M_{S^1}^{(N)}\to \R$ is Morse. We denote by $S_p$, $p\in 
\textrm{Crit}(a)$ these circles 
of critical points and by $[p]\in M_{S^1}^{(N)}$ the 
nondegenerate critical point of $\underline a$ corresponding to $S_p$, 
so that $S_p = S_{\tau\cdot p}$ and $[p]=[\tau\cdot p]$, 
$\tau\in S^1$. We denote by  
$$ 
\ind(S_p) = \ind([p]) 
$$ 
the Morse-Bott index of $S_p$.  
 
We choose 
a generic $S^1$-invariant metric $g$ on $M\times S^{2N+1}$ 
such that the gradient flow of $a$ has the Thom-Smale transversality 
property, i.e.    
$$ 
W^u(S_p) \pitchfork W^s(S_q), \quad p,q\in \textrm{Crit}(a). 
$$ 
This is equivalent to asking that the gradient flow of 
$\underline a$ with respect to the induced metric $\underline g$ on 
$M_{S^1}^{(N)}$ satisfies $W^u([p]) \pitchfork W^s([q])$, $[p],[q]\in 
\textrm{Crit}(\underline a)$.  
Given $\op,\up\in \textrm{Crit}(a)$ we denote by  
$$ 
\widehat \cM(S_\op,S_\up; a,g) 
$$ 
the {\bf space of gradient trajectories} consisting of maps 
$v=(u,\lambda):\R\to M\times S^{2N+1}$ which satisfy  
\begin{equation} 
\dot v = -\vec \nabla a (v)  
\quad \Leftrightarrow \quad  
\left\{\begin{array}{rcl} 
 \dot u & = & - \vec \nabla _x  a (u,\lambda), \\ 
 \dot \lambda & = & - \vec \nabla _\lambda  a(u,\lambda), 
\end{array}\right. 
\end{equation} 
and 
\begin{equation}  
\left\{\begin{array}{rcl}  
\displaystyle 
\lim_{s\to -\infty} v(s) & \in & S_\op, \\ 
\displaystyle 
\lim_{s\to \infty} v(s) & \in & S_\up, 
\end{array}\right.  
\quad \Leftrightarrow \quad  
\left\{\begin{array}{rcl} 
\displaystyle 
\lim_{s\to -\infty} (u(s),\lambda(s)) & = & (\overline x,\overline 
\lambda)\in S_\op, \\  
\displaystyle 
\lim_{s\to \infty} (u(s),\lambda(s)) & = & (\underline x,\underline 
\lambda)\in S_\up.  
\end{array}\right.   
\end{equation} 
Here the gradient $\vec \nabla$ is considered with respect to the metric 
$g$ and $\vec \nabla_x, \vec \nabla_\lambda$ are its components along $TM$ 
and $TS^{2N+1}$ respectively.  
Under the transversality assumption for 
the metric $g$ the space of gradient trajectories is a smooth manifold 
of dimension   
$$ 
\dim \, \widehat \cM(S_\op,S_\up; a,g) =  
\ind(S_\op) - \ind(S_\up) +1.  
$$ 
It carries a natural action of $\R$ by reparametrization and we 
denote by 
$$ 
\cM(S_\op,S_\up; a,g) := \widehat \cM(S_\op,S_\up; a,g) /\R 
$$ 
the {\bf moduli space of gradient trajectories}. In our setting the 
moduli space carries an action of $S^1$ and the quotient 
$$ 
\cM_{S^1}(S_\op,S_\up; a,g) := \cM(S_\op,S_\up; a,g)/S^1 
$$ 
is a smooth manifold of dimension  
$$ 
\dim \, \cM_{S^1}(S_\op,S_\up; a,g)= \ind(S_\op)-\ind(S_\up)-1. 
$$ 
 
The bundle with fiber $TW^u(\tau\cdot p)$, $\tau\in S^1$ over 
$S_p$ is orientable since $W^u(S_p):=\bigcup_{\tau\in S^1} 
W^u(\tau\cdot p)$ carries an action of $S^1$. We choose for each 
$S_p$ an orientation of this bundle, which amounts to choosing an 
orientation of $W^u(S_p)$. Since each $S_p$ inherits a natural 
orientation from $S^1$, this determines a coorientation of the bundle  
with fiber $TW^s(\tau\cdot p)$, $\tau\in S^1$ over 
$S_p$ and therefore a coorientation of $W^s(S_p):=\bigcup_{\tau\in S^1} 
W^s(\tau\cdot p)$. We get orientations on $\widehat 
\cM(S_\op,S_\up;a,g)$ and, after quotienting out $\R$ and $S^1$, we 
get orientations on $\cM_{S^1}(S_\op,S_\up;a,g)$, $\op,\up\in 
\textrm{Crit}(a)$. In particular, if $\ind(\op)-\ind(\up)=1$ the 
moduli space $\cM_{S^1}(S_\op,S_\up;a,g)$ is zero-dimensional and each 
element $[v]$ inherits a sign $\epsilon([v])$.  
 
We define the {\bf \boldmath$S^1$-equivariant Morse complex} by  
$$ 
C_k^{S^1} (a,g) := \bigoplus _{\ind(S_p)=k} \Z \langle S_p \rangle, 
$$ 
with the {\bf \boldmath$S^1$-equivariant Morse differential}  
$$ 
d^{S^1} : C_k^{S^1} \to C_{k-1} ^{S^1}, 
$$ 
$$ 
d^{S^1}\langle S_\op \rangle := \sum_{\ind(S_\op)-\ind(S_\up)=1 \ }  
\sum_{\ [v]\in \cM_{S^1}(S_\op,S_\up;a,g)} \epsilon([v]) \langle S_\up 
\rangle.   
$$ 
 
Since the elements of $\cM_{S^1}(S_\op,S_\up; a,g)$ are in one-to-one 
correspondence with elements of the moduli space 
$\cM([\op],[\up];\underline a, \underline g)$ of gradient trajectories 
of $\underline a$ with respect to the metric $\underline g$ on 
$M_{S^1}^{(N)}$, and since the rule for obtaining signs on 
$\cM_{S^1}(S_\op,S_\up;a,g)$ if $\ind(S_\op)-\ind(S_\up)=1$ 
induces the usual Morse homology rule for signs on 
$\cM([\op],[\up];\underline a, \underline g)$, we infer that the 
complex $(C_*^{S^1},d^{S^1})$ is tautologically isomorphic with the 
Morse complex of the pair $(\underline a,\underline g)$. Therefore 
$$ 
H_k(C_*^{S^1},d^{S^1}) \simeq H_k(M_{S^1}^{(N)}), \quad k\in \N 
$$ 
and, for $N$ large enough (depending on $k$), we have  
$$ 
H_k(C_*^{S^1},d^{S^1}) \simeq H_k^{S^1}(M). 
$$

\begin{remark} {\rm   
The previous construction admits an obvious reformulation for any 
manifold $P$ endowed with a free $S^1$-action: the homology of the 
quotient $P/S^1$ can be described in terms of Morse-Bott data on $P$ 
alone.  
} 
\end{remark}

\subsection[$S^1$-equivariant symplectic homology]{\boldmath$S^1$-equivariant symplectic homology}  
\label{sec:S1equivsymplhom}  
 
In this section we give the definition of $S^1$-equivariant symplectic 
homology following Viterbo~\cite{V}. Our treatment parallels the 
finite dimensional case as presented in~\S\ref{sec:S1equivhom}. We obtain the definition of 
$S^1$-equivariant symplectic homology as a variant of parametrized 
symplectic homology with $\Lambda=S^{2N+1}$.  
 
The space of smooth loops $\gamma:S^1\to \widehat W$ carries an action 
of $S^1$ given by   
$$ 
(\tau \cdot \gamma)(\cdot) := \gamma(\cdot -\tau), \quad \tau \in S^1.  
$$ 
 
Let $H:S^1 \times \widehat W \times S^{2N+1} \to \R$ be a family of 
Hamiltonian functions denoted by 
$H(\theta,x,\lambda)=H_\lambda(\theta,x)$. This defines a family of 
action functionals    
$$ 
\cA : C^\infty(S^1,\widehat W)\times S^{2N+1} \to \R, 
$$ 
$$ 
\cA(\gamma,\lambda) = \cA_\lambda(\gamma) := -\int_{[0,1]\times S^1} \sigma^* 
\om - \int_{S^1} H_\lambda(\theta,\gamma(\theta)) d\theta,  
$$ 
where $\sigma:[0,1]\times S^1\to \widehat W$ is a smooth homotopy from 
$l_{[\gamma]}$ to $\gamma$, and $l_{[\gamma]}$ is a fixed 
representative of the free homotopy class of $\gamma$.

\begin{lemma} 
 The family $\cA$ is invariant with respect to the diagonal action 
of $S^1$ if and only if the family of Hamiltonians satisfies  
\begin{equation} \label{eq:H} 
H_{\tau\lambda}(\theta+\tau,\cdot)=H_\lambda(\theta,\cdot) + 
r(\theta,\tau,\lambda)  
\end{equation}  
for some function $r:S^1\times S^1\times S^{2N+1} \to \R$ such that  
\begin{equation} \label{eq:h1}  
\int_{S^1} r(\theta,\tau,\lambda) d\theta =0 \mbox{ for all } \tau\in 
S^1,\ \lambda\in S^{2N+1} 
\end{equation}  
and 
\begin{equation} \label{eq:h2} 
r(\theta,1,\lambda) = 0, \qquad 
r(\theta+\tau,-\tau,\tau\lambda)=-r(\theta,\tau,\lambda).  
\end{equation}  
\end{lemma}  
 
\begin{proof} 
The nontrivial implication is that invariance of $\cA$ implies the 
desired condition on $H$. We thus assume that $\cA$ is 
invariant, i.e. $\cA_{\tau\lambda} 
(\tau\gamma) = \cA_\lambda(\gamma)$ for all loops $\gamma$. This 
is equivalent to the equality  
$$ 
\int_{S^1} H_{\tau\lambda}(\theta+\tau,\gamma(\theta))d\theta 
= \int_{S^1} H_\lambda(\theta,\gamma(\theta)) d\theta, \ \forall \ \gamma 
$$ 
and, denoting 
$F(\theta,\tau,\lambda,x):=H_{\tau\lambda}(\theta+\tau,x)-H_\lambda(\theta,x)$, we 
obtain   
$$ 
\int_{S^1} F(\theta,\tau,\lambda,\gamma(\theta))d\theta=0, \ \forall \ \gamma,\tau,\lambda. 
$$ 
By letting $\gamma$ vary in the neighbourhood of the constant loop at 
some $x\in \widehat W$ we see that 
we must have $\int_{S^1} D_x F(\theta,\tau,\lambda,x)\cdot \zeta(\theta) d\theta =0$ for 
all loops $\zeta$ of tangent vectors at $x$. It follows that $D_x 
F(\theta,\tau,\lambda,x)=0$ for all $\theta\in S^1$ and, since $x$ was chosen 
arbitrarily, we get $F(\theta,\tau,\lambda,x)=r(\theta,\tau,\lambda)$ 
with $\int_{S^1} r(\theta,\tau,\lambda) d\theta =0$. This 
shows~\eqref{eq:h1}, whereas~\eqref{eq:h2} is straightforward.   
\end{proof}  
 
\begin{remark} {\rm  
 Condition~\eqref{eq:H} holds for example if $r\equiv 0$, 
i.e. if the family $H$ satisfies 
\begin{equation} \label{eq:HS1inv}
H_{\tau\lambda}(\theta+\tau,\cdot)=H_\lambda(\theta,\cdot). 
\end{equation}
In particular one can choose the family $H$ to be given by a 
single autonomous Hamiltonian $H(\theta,x,\lambda)=H(x)$.  
} 
\end{remark} 
 
We denote by $\cH^{S^1}_N\subset \cH_{S^{2N+1}}$ the set of admissible   
Hamiltonian families $H:S^1\times \widehat W\times S^{2N+1}\to \R$ 
satisfying condition~\eqref{eq:HS1inv}. It follows from the 
definitions that there exists $t_0\ge 0$ such that, for $t\ge t_0$, we 
have $H(\theta,p,t,\lambda)=\beta e^t +\beta'(\lambda)$, with 
$0<\beta\notin\mathrm{Spec}(M,\alpha)$, and $\beta'\in 
C^\infty(S^{2N+1},\R)$ invariant under the action of $S^1$.  
 
The differential of $\cA$ is given by~\eqref{eq:dA} and critical 
points of $\cA$ satisfy~\eqref{eq:periodicpar}. Since $\cA$ is 
$S^1$-invariant, the set of critical points of $\cA$ is 
$S^1$-invariant as well,   
i.e. if $(\gamma,\lambda)\in\cP(H)$, then $(\tau\gamma,\tau\lambda)\in 
\cP(H)$ for all $\tau\in S^1$. Given 
$p:=(\gamma,\lambda)\in \cP(H)$ we denote  
$$ 
S_p=S_{(\gamma,\lambda)}:= \{(\tau\gamma,\tau\lambda) \ : \ 
\tau\in S^1\} \subset \cP(H), 
$$ 
so that $S_p=S_{\tau \cdot p}$, $\tau\in S^1$. We shall refer to $S_p$ 
as an {\bf \boldmath$S^1$-orbit of critical points} (of $\cA$).  
 
An {\bf admissible family of almost complex structures} 
$J=(J_\lambda^\theta)$ (in the sense of Section~\ref{sec:param}) is 
called {\bf \boldmath$S^1$-invariant} if it satisfies the condition    
\begin{equation} \label{eq:J} 
J_{\tau\lambda}^{\theta+\tau}=J_\lambda^\theta, \qquad \theta\in S^1,\ 
\tau\in S^1, \ \lambda\in S^{2N+1}.  
\end{equation} 
Such a $J^\theta$ induces an $S^{2N+1}$-family of $L^2$-metrics on 
$C^\infty(S^1,\widehat W)$ defined by  
$$ 
\langle \zeta,\eta\rangle_\lambda := \int_{S^1} 
\om(\zeta(\theta),J_\lambda^\theta\eta(\theta)) d\theta, \quad \zeta,\eta\in 
T_\gamma C^\infty(S^1,\widehat 
W)=\Gamma(\gamma^*T\widehat W). 
$$ 
Condition~\eqref{eq:J} ensures that, when coupled with an 
$S^1$-invariant metric $g$ on $S^{2N+1}$, this family gives 
rise to an $S^1$-invariant metric on 
$C^\infty(S^1,\widehat W) \times S^{2N+1}$. We denote by $\cJ_N^{S^1}$ 
the set of pairs $(J,g)$ consisting of an $S^1$-invariant admissible 
family of almost complex structures $J$ on $\widehat W$ and of an 
$S^1$-invariant Riemannian metric $g$ on $S^{2N+1}$.  
 
Given  
$H\in\cH^{S^1}_N$, $(J,g)\in\cJ^{S^1}_N$, and 
$\op:=(\og,\olambda),\up:=(\ug,\ulambda)\in \cP(H)$, we denote by  
$$ 
\widehat \cM(S_\op,S_\up;H,J,g) 
$$ 
the {\bf space of \boldmath$S^1$-equivariant Floer trajectories}, consisting of 
pairs $(u,\lambda)$ with  
$$ 
u:\R\times S^1 \to \widehat W, \qquad \lambda:\R\to S^{2N+1},  
$$ 
satisfying  
\begin{eqnarray}  
\label{eq:Floer1} 
 \p_s u + J_{\lambda(s)}^\theta \p_\theta u - 
J_{\lambda(s)}^\theta X_{H_{\lambda(s)}}^\theta (u) & = & 0, \\ 
\label{eq:Floer2} 
 \dot \lambda (s) - \int_{S^1} \vec \nabla_\lambda 
H(\theta,u(s,\theta),\lambda(s)) d\theta & = & 0,   
\end{eqnarray} 
and  
\begin{equation} \label{eq:asymptotic} 
 \lim_{s\to -\infty} (u(s,\cdot),\lambda(s)) \in S_\op, \quad  
 \lim_{s\to +\infty} (u(s,\cdot),\lambda(s)) \in S_\up. 
\end{equation} 
 
The additive group $\R$ acts on $\widehat \cM(S_\op,S_\up;H,J,g)$ 
by reparametrization in the $s$-variable. We denote by  
$$ 
\cM(S_\op,S_\up;H,J,g) := \widehat \cM(S_\op,S_\up;H,J,g)/\R 
$$ 
the {\bf moduli space of \boldmath$S^1$-equivariant Floer trajectories}.  
This space is endowed with natural evaluation maps  
$$ 
\oev:\cM(S_\op,S_\up;H,J,g)\to S_\op, \qquad 
\uev:\cM(S_\op,S_\up;H,J,g)\to S_\up. 
$$ 
 
An $S^1$-orbit of critical points $S_p\subset \cP(H)$ is called {\bf 
  nondegenerate} if the Hessian $d^2\cA(\gamma,\lambda)$ has a 
  $1$-dimensional kernel $V_p$ for some (and hence any) $(\gamma,\lambda)\in 
  S_p$. It follows from~\cite[Lemma~2.3]{BOtrans}
  that nondegeneracy is equivalent to the fact 
  that the kernel of the asymptotic operator $D_p$  
  defined in~\eqref{eq:Dasy} 
  is also 
  $1$-dimensional and equal to $V_p$. In both cases, a generator of 
  $V_p$ is given by the infinitesimal generator of the $S^1$-action.  
 
We define the set $\cH^{S^1}_{N,\reg}\subset \cH^{S^1}_N$ to consist 
of elements $H$ such that, for any $p\in\cP(H)$, the $S^1$-orbit $S_p$ 
is nondegenerate. We proved in~\cite[Proposition~5.1]{BOtrans} that 
 the set $\cH^{S^1}_{N,\reg}$ is of the second Baire category in 
 $\cH^{S^1}_N$. Moreover, if $H\in \cH^{S^1}_{N,\reg}$, each 
 $S^1$-orbit $S_p\subset C^\infty(S^1,\widehat W)\times S^{2N+1}$ is 
 isolated.

Let $d>0$ be small enough (for a fixed $H\in\cH^{S^1}_{N,\reg}$, one 
can take $d>0$ to be smaller than the minimal spectral gap of the asymptotic 
operators $D_p$, $p\in\cP(H)$), and fix $1<p<\infty$. Given 
$\op,\up\in \cP(H)$ and $(u,\lambda)\in \widehat 
\cM(S_\op,S_\up;H,J,g)$, we define   
\begin{eqnarray*} 
  \cW^{1,p,d} & := & W^{1,p}(u^*T\widehat 
  W;e^{d|s|}dsd\theta) \oplus W^{1,p}(\lambda^* 
  TS^{2N+1};e^{d|s|}ds)\oplus V_{\op}\oplus V_{\up}, \\ 
 \cL^{p,d} & := & L^p(u^*T\widehat 
  W;e^{d|s|}dsd\theta) \oplus L^p(\lambda^* 
  TS^{2N+1};e^{d|s|}ds). 
\end{eqnarray*} 
Here we identify $V_{\op}$, $V_{\up}$ with the $1$-dimensional spaces  
generated by the sections $\beta(s)(\dot\og,X_{\olambda})$, respectively $\beta(-s)(\dot\ug,X_{\ulambda})$ 
 of $u^*T\widehat W\oplus \lambda^*TS^{2N+1}$. For this identification, we denote by $X_{\olambda}$, $X_{\ulambda}$ the values of the infinitesimal generator of the $S^1$-action on $S^{2N+1}$  
 at the points $\olambda$, respectively $\ulambda$, and choose a cut-off function $\beta:\R\to [0,1]$ which is equal to $1$ near $-\infty$, and vanishes near $+\infty$. 
 For the next proposition we recall the definition of the linearized operator $D_{(u,\lambda)}$ in~\eqref{eq:Dulambda}.
 
\begin{proposition} \label{prop:indexMB} 
Assume $S_\op,S_\up\subset \cP(H)$ are nondegenerate. For any 
$(u,\lambda)\in \widehat \cM(S_\op,S_\up;H,J,g)$ the operator  
$$ 
D_{(u,\lambda)}: \cW^{1,p,d} \to \cL^{p,d} 
$$ 
is Fredholm of index  
$$ 
\ind\, D_{(u,\lambda)} = -\mu(\op) +\mu(\up) + 1. 
$$ 
\end{proposition}  
 
In the above statement, it is understood that the trivialization used 
to define $\mu(\op)$ is obtained from the trivialization used 
to define $\mu(\up)$ by continuation along the map $u$. 
 
\begin{proof} 
The Fredholm property was proved in~\cite[Proposition~5.2]{BOtrans} as follows. 
Let $\cW^{1,p}$ and $\cL^p$ be defined as $\cW^{1,p,d}$ and $\cL^{p,d}$ above, with $d=0$ and without taking into account the direct summands $V_{\op}$, $V_{\up}$. Let $\widetilde D_{(u,\lambda)}:\cW^{1,p}\to \cL^p$ be the operator obtained by conjugating with $e^{\frac d p |s|}$ the restriction of $D_{(u,\lambda)}$ to  
$W^{1,p}(u^*T\widehat   W;e^{d|s|}dsd\theta) \oplus W^{1,p}(\lambda^* TS^{2N+1};e^{d|s|}ds)$.  
It follows from our choice of $d>0$ that 
$\widetilde D_{(u,\lambda)}$ has nondegenerate asymptotics, hence it is Fredholm by~\cite[Theorem~2.6]{BOtrans}. Since the restriction of $D_{(u,\lambda)}$ to a codimension $2$ subspace is conjugate to $\widetilde D_{(u,\lambda)}$, it follows that $D_{(u,\lambda)}$ is Fredholm as well. 

The asymptotic operator at $-\infty$ is $\widetilde D_{\op}=D_{\op}+\frac d p \one$, and the asymptotic operator at $+\infty$ is $\widetilde D_{\up}=D_{\up}-\frac d p \one$.  The indices after perturbation are given by $\mu(\op)+\frac 12$, respectively $\mu(\up)-\frac 12$, and
using the Main Theorem in~\cite{BOparam} we obtain
\begin{eqnarray*}
\ind \, D_{(u,z)} & = & \ind \, \widetilde D_{(u,z)} + 2 \\
& = & -(\mu(\op)+\frac 1 2) + (\mu(\up)-\frac 1 2) + 2 \\
& = & -\mu(\op) +\mu(\up) + 1.
\end{eqnarray*}
\end{proof}  
 
Let $H\in\cH^{S^1}_{N,\reg}$. A pair $(J,g)\in \cJ^{S^1}_N$ is called 
{\bf regular for \boldmath$H$} if the operator $D_{(u,\lambda)}$ is surjective 
for any $\op,\up\in\cP(H)$ and any $(u,\lambda)\in \widehat 
\cM(\op,\up;H,J,g)$. We denote the set of such regular pairs by 
$\cJ^{S^1}_{N,\reg}(H)$. 

We defined in~\cite[\S7]{BOtrans} two special classes $\cH_*\cJ'\subset \cH\cJ'$ in $\cH_N^{S^1}\times\cJ_N^{S^1}$. We proved in~\cite[Theorem~7.4]{BOtrans} that 
there exists an open subset $\cH\cJ'_{\mathrm{reg}}\subset \cH\cJ'$
which is dense in a neighbourhood of $\cH_*\cJ'\subset \cH\cJ'$, and
which consists of triples $(H,J,g)$ such that  
$$
H\in\cH^{S^1}_{N,\mathrm{reg}},\qquad (J,g)\in\cJ^{S^1}_{N,\mathrm{reg}}(H).
$$

Let $(H,J,g)\in \cH\cJ'_{\reg}$.  
Recall that, for each $p=(\gamma,\lambda)\in\cP(H)$, we have chosen a 
cylinder $\sigma_p:[0,1]\times S^1\to\widehat W$ such that 
$\sigma_p(0,\cdot)=l_{[\gamma]}$ and $\sigma_p(1,\cdot)=\gamma$. We 
define $\overline \sigma_p(s,\theta):=\sigma_p(1-s,\theta)$. Given 
$\op=(\og,\olambda),\up=(\ug,\ulambda)\in\cP(H)$ we define  
$$ 
\cM^A(S_\op,S_\up;H,J,g)\subset \cM(S_\op,S_\up;H,J,g) 
$$ 
to consist of trajectories $(u,\lambda)$ such that 
$[\sigma_{\op}\#u\#\overline \sigma_{\up}]=A\in H_2(\widehat 
W;\Z)$. It follows from Proposition~\ref{prop:indexMB} that  
\begin{equation} \label{eq:dimMS} 
\dim\, \cM^A(S_\op,S_\up;H,J,g) = -\mu(\op) +\mu(\up) +2\langle 
c_1(T\widehat W),A\rangle. 
\end{equation}  
Since $\cA$ and $(J,g)$ are $S^1$-invariant,  
the moduli space $\cM^A(S_\op,S_\up;H,J,g)$ 
carries a free action of $S^1$ induced by the diagonal 
action on $C^\infty(S^1,\widehat W) \times S^{2N+1}$, i.e. 
$$ 
\tau\cdot (u,\lambda) := (u(\cdot,\cdot-\tau),\tau\lambda).  
$$ 
We denote the quotient by  
$$ 
\cM_{S^1}(S_\op,S_\up;H,J,g) := \cM(S_\op,S_\up;H,J,g)/S^1. 
$$ 
This is a smooth manifold of dimension  
$$ 
\dim\, \cM^A_{S^1}(S_\op,S_\up;H,J,g) = -\mu(\op) +\mu(\up) +2\langle 
c_1(T\widehat W),A\rangle -1. 
$$ 
 
\begin{remark} \label{rmk:coherent}
An important feature of these moduli spaces is that they admit a system  
of coherent orientations in the sense of~\cite{FH}. The difference with respect to  
the setup of Floer homology is that the asymptotes for the moduli spaces are not fixed, but can vary along circles $S_p$, $p=(\gamma,\lambda)\in\cP(H)$. However, if one chooses the trivializations of $\gamma^*T\widehat W \oplus T_\lambda S^{2N+1}$ so that they are invariant under the $S^1$-action,  
then the analytical expression of the asymptotic operators $D_p$, $p\in\cP(H)$ only depends on $S_p$.  
It then follows from the arguments in~\cite{FH} that the spaces of Fredholm operators of the  
form~\eqref{eq:Dtriv} with nondegenerate asymptotics of the form  
$D_p$, $p\in\cP(H)$ are contractible, and hence the corresponding determinant line bundles are orientable.  
The system of coherent orientations on the moduli spaces $ \cM^A_{S^1}(S_\op,S_\up;H,J,g)$ is obtained by pulling back a system of coherent orientations on these spaces of Fredholm operators, as in~\cite{FH}. This implies that all the moduli spaces involved are orientable and hence, unlike the situation of symplectic field theory, there is no notion of good and bad $S^1$-orbit in the context of $S^1$-equivariant symplectic homology. 
\end{remark}
 
Given a free homotopy class $a$ in $\widehat W$, we define 
the {\bf \boldmath$S^1$-equivariant chain complex} $SC^{a,S^1,N}_*(H,J,g)$ as a 
chain complex whose underlying $\Lambda_\omega$-module is  
\begin{equation} \label{eq:SCS1} 
SC^{a,S^1,N}_*(H):=SC^{a,S^1,N}_*(H,J,g):= 
\bigoplus_{S_p\subset \cP^a(H)} \Lambda_\omega\langle 
S_p\rangle.  
\end{equation}  
The grading is defined by 
$$
|S_p\, e^A| := -\mu(p) +N -2\langle 
c_1(T\widehat W),A\rangle.
$$ 
(The reason for introducing a shift by $N$ will become apparent in the proof of Lemma~\eqref{lem:minus} below.)
We define the {\bf \boldmath$S^1$-equivariant differential} 
$\partial^{S^1}:SC^{a,S^1,N}_*(H)\to SC^{a,S^1,N}_{*-1}(H)$ 
by 
$$ 
\partial^{S^1}(S_\op):=\sum_{\substack{ 
  S_\up\subset \cP^a(H) \\ 
|S_\op| - |S_\up\, e^A|=1}} 
\ \sum_{\scriptstyle [u]\in \cM^A_{S^1}(S_\op,S_\up;H,J,g)} 
\epsilon([u])S_\up\, e^A. 
$$ 
The sign $\epsilon([u])$ is obtained by comparing the coherent 
orientation of the moduli space 
$\cM^A_{S^1}(S_\op,S_\up;H,J,g)$ with the orientation induced by the 
infinitesimal generator of the $S^1$-action.  
 
\begin{proposition} \label{prop:partialS1}  
The map $\partial^{S^1}$ satisfies  
$$ 
\partial^{S^1}\circ \partial^{S^1}=0. 
$$ 
\end{proposition}  
 
The proof of Proposition~\ref{prop:partialS1} is given in 
Section~\ref{sec:MBparam}.  
We define the {\bf \boldmath$S^1$-equivariant Floer homology groups} by 
$$ 
SH^{a,S^1,N}_*(H,J,g):=H_*(SC^{a,S^1,N}_*(H),\partial^{S^1}). 
$$ 
 
\begin{proposition} \label{prop:indepJg}  
Let $H\in\cH^{S^1}_{N,\reg}$. Given 
$(J_1,g_1),(J_2,g_2)\in\cJ^{S^1}_{N,\reg}(H)$, there exists a 
canonical isomorphism  
$$ 
SH^{a,S^1,N}_*(H,J_1,g_1) \simeq SH^{a,S^1,N}_*(H,J_2,g_2). 
$$ 
\end{proposition}  
 
We prove Proposition~\ref{prop:indepJg} in  
Section~\ref{sec:continuation}. Given $H\in\cH^{S^1}_{N,\reg}$ we shall denote  
$SH^{a,S^1,N}_*(H):=SH^{a,S^1,N}_*(H,J,g)$ for $(J,g)\in 
\cJ^{S^1}_{N,\reg}(H)$. In analogy with the construction of symplectic 
homology, we define   
$$ 
SH^{a,S^1,N}_*(W):=\lim_{\stackrel \longrightarrow {H\in 
  \cH^{S^1}_{N,\reg}}} SH^{a,S^1,N}_*(H). 
$$ 
The {\bf \boldmath$S^1$-equivariant symplectic homology groups of \boldmath$W$} are 
defined by  
$$ 
SH^{a,S^1}_*(W):=\lim_{\stackrel \longrightarrow N} SH^{a,S^1,N}_*(W). 
$$ 
The direct limit is taken with respect to the embeddings 
$S^{2N+1}\hookrightarrow S^{2N+3}$, inducing maps  
$SH^{a,S^1,N}_*(W)\to SH^{a,S^1,N+1}_*(W)$ (see 
Remark~\ref{rmk:SHincl}).  
 
For the particular case of the trivial homotopy class $a=0$, we denote 
the $S^1$-equivariant symplectic homology groups by $SH^{S_1}_*(W)$.  
Given $H\in\cH^{S^1}_{N,\reg}$ we define the {\bf parametrized reduced 
  action functional} $\cA^0:C^\infty_{\mathrm{contr}}(S^1,\widehat 
W)\times S^{2N+1}\to\R$ by  
$$ 
\cA^0(\gamma,\lambda):=-\int_{D^2}\sigma^*\widehat \omega - \int_{S^1} 
H(\theta,\gamma(\theta),\lambda)\, d\theta. 
$$ 
Here $\sigma:D^2\to\widehat W$ is a smooth extension of $\gamma$, and 
$\cA^0$ is well-defined due to assumption~\eqref{eq:asph}.  
 
Similarly to the case of symplectic homology, we define a special 
cofinal class of Hamiltonian families $\cH^{\prime\, S^1}_N\subset 
\cH^{S^1}_N$, consisting of elements $H=(H_\lambda)\in\cH^{S^1}_N$ 
such that $H_\lambda\in \cH'$ for all $\lambda\in S^{2N+1}$ (see 
Section~\ref{sec:symplhom} for the definition of the class $\cH'$).  
 
Given $H\in\cH^{\prime\, S^1}_{N,\reg}:=\cH^{\prime\, S^1}_N\cap 
\cH^{S^1}_{N,\reg}$, $(J,g)\in\cJ^{S^1}_{N,\reg}(H)$, and $\eps>0$ 
small enough, we define the chain complexes
$$ 
SC^{-,S^1,N}_*(H,J,g):=\bigoplus _{\substack{S_p\subset\cP^0(H) 
    \\ \cA^0(p)\le \eps}} \Lambda_\omega \langle S_p\rangle \subset 
SC^{S^1,N}_*(H,J,g) 
$$ 
and  
$$ 
SC^{+,S^1,N}_*(H,J,g):=SC^{S^1,N}_*(H,J,g)/SC^{-,S^1,N}_*(H,J,g). 
$$ 
The differential on $SC^{\pm,S^1,N}_*(H,J,g)$ is induced by 
$\partial^{S^1}$. The corresponding homology groups 
$SH^{\pm,S^1,N}_*(H,J,g)$ do not depend on $(J,g)$ and $\eps$, and we 
define  
$$ 
SH^{\pm,S^1,N}_*(W):=\lim_{\stackrel \longrightarrow 
  {H\in\cH^{\prime\, S^1}_{N,\reg}}}SH^{\pm,S^1,N}_*(H,J,g). 
$$ 
Passing to the direct limit over $N\to\infty$, we define  
$$ 
SH^{\pm,S^1}_*(W):=\lim_{\stackrel \longrightarrow 
  {N}}SH^{\pm,S^1,N}_*(W). 
$$ 
We call $SH^{+,S^1}_*(W)$ the {\bf positive \boldmath$S^1$-equivariant 
  symplectic homology group} of $(W,\omega)$. It follows from the 
definitions that this fits into the {\bf tautological long exact 
  sequence}  
$$ 
\dots \to SH^{+,S^1}_{k+1}(W) \to SH^{-,S^1}_k(W)\to SH^{S^1}_k(W) \to 
SH^{+,S^1}_k(W)\to \dots  
$$ 
 
\begin{lemma} \label{lem:minus} Assume $W$ has positive contact type boundary in the 
  sense of Section~\ref{sec:symplhom}. There is a natural isomorphism  
$$ 
SH^{-,S^1}_*(W) \simeq H^{S^1}_{*+n}(W,\partial W;\Lambda_\omega). 
$$ 
Here $H^{S^1}_{*+n}(W,\partial W;\Lambda_\omega)\simeq 
H_{*+n}(W,\partial W;\Lambda_\omega)\otimes H_*(\C P^\infty;\Z)$ 
denotes the $S^1$-equivariant homology of the pair $(W,\partial W)$ 
with respect to the trivial $S^1$-action. 
\end{lemma}  
 
\begin{proof} 
We consider a Hamiltonian $H\in\cH^{\prime\, S^1}_{N,\reg}$ which has
the form
\begin{equation} \label{eq:split}
H(\theta,x,\lambda)=K(x)+\widetilde f(\lambda)
\end{equation}
on $S^1\times W\times S^{2N+1}$, with $K:W\to \R$ a $C^2$-small
function, and $\widetilde f:S^{2N+1}\to \R$ the lift of a Morse
function $f:\C P^N\to\R$. We choose $(J,g)\in
\cJ^{S^1}_{N,\reg}(H)$ such that $J$ is independent of $\theta$ and
$\lambda$ on $W$. To find such a pair we use that $H$ has the split form~\eqref{eq:split} and the manifold $W$ is symplectically aspherical. Firstly, by~\cite[Theorems~7.3 and~8.1]{SZ} we can find a generic such $J$ which is regular for the Floer equation involving $K$ on $W$. Secondly, since the Hamiltonian $H$ is split and independent of $\theta$, the PDE system for the parametrized Floer trajectories is split as well, the second equation~\eqref{eq:Floer2par} in this system reduces to the 
negative 
gradient flow equation for 
$-\widetilde f$,
and we can therefore choose a generic $S^1$-invariant $g$ which is regular for the latter. Then $(J,g)\in\cJ^{S^1}_{N,\reg}(H)$. 

Since the
parametrized Floer equation is split and the Floer complex for $(K,J)$
reduces to the 
homological 
Morse complex
for $-K$,
we have an isomorphism of complexes
\begin{equation} \label{eq:ggg}
SC^{-,S^1,N}_*(H,J,g)=
C_{*+n}(-K,J;\Lambda_\omega)\otimes C^{S^1}_*(-\widetilde f,g;\Z).
\end{equation}
Here $C_*$ denotes the corresponding Morse complexes, and our convention for the grading again plays a role. Since $C^{S^1}_*(\widetilde f,g;\Z)$ corresponds to a Morse complex on $\C P^N$, the conclusion follows.

To see that the grading in~\eqref{eq:ggg} is correct, let us consider a critical point 
$p=(\gamma,\lambda)$,
with $\gamma$ a constant orbit of $X_K$ at a critical point, still denoted $\gamma$. Let $\mu_{RS}$ denote the Robbin-Salamon index, and $\mathrm{ind}(q,\phi)$ the index of a critical point $q$ of a function $\phi$. Using the \emph{(Splitting)} axiom in~\cite[Prop.~4]{BOparam}, we obtain $\mu(p)=\mu_{RS}(\gamma)+\frac 1 2 \mathrm{sign}\, \mathrm{Hess}_\lambda(-\tf) = n-\mathrm{ind}(\gamma;-K)+\frac 1 2 (2N-2\,\mathrm{ind}(\lambda;-\tf)) = n-\mathrm{ind}(\gamma;-K)+ N-\mathrm{ind}(\lambda;-\tf)$, so that $\mathrm{ind}(\gamma;-K)+\mathrm{ind}(\lambda;-\tf)=|p|+n$. 
\end{proof}


\section{Morse-Bott constructions} \label{sec:MB}  
 
\subsection{Morse-Bott complex for parametrized symplectic homology} 
  \label{sec:MBparam}

We describe in this section a Morse-Bott construction for parametrized 
symplectic homology in the case when $\Lambda=S^{2N+1}$ and the action 
functional $\cA:C^\infty(S^1,\widehat W)\times S^{2N+1}\to \R$ is 
$S^1$-invariant with respect to the diagonal action of $S^1$. The 
situation is analogous to that of Floer homology for an autonomous 
Hamiltonian considered in~\cite{BOauto}.

Let $H\in\cH^{S^1}_{N,\reg}$ and $(J,g)\in\cJ^{S^1}_{N,\reg}(H)$ as in 
Section~\ref{sec:S1equivsymplhom}. For each $S^1$-orbit of critical 
points $S_p\subset \cP(H)$ we choose a perfect Morse function 
$f_p:S_p\to\R$. We denote by $m_p$, $M_p$ the minimum, respectively 
the maximum of $f_p$. Given $\op,\up\in\cP(H)$, 
$Q_\op\in\mathrm{Crit}(f_\op)$, $Q_\up\in\mathrm{Crit}(f_\up)$, and 
$m\ge 0$, we denote by   
$$ 
\cM^A_m(Q_\op,Q_\up;H,\{f_p\},J,g) 
$$ 
the union for $p_1,\dots,p_{m-1}\in\cP(H)$ and $A_1+\dots+A_m=A$ of 
the fibered products  
\begin{eqnarray*} 
&& 
\hspace{-1cm}W^u(Q_\op)  
\times_{\oev} 
(\cM^{A_1}(S_{\op}\,,S_{p_1})\!\times\!\R^+) 
{_{\varphi_{f_{p_1}}\!\circ\uev}}\!\times   
_{\oev} 
(\cM^{A_2}(S_{p_1},S_{p_2})\!\times\!\R^+) \\ 
&&  
{_{\varphi_{f_{p_2}}\!\circ\uev}\times_{\oev}} \ldots\, 
{_{\varphi_{f_{p_{m-1}}}\!\!\circ\uev}}\!\!\times 
_{\oev}   
\cM^{A_m}(S_{p_{m-1}},\!S_{\up})  
{_{\uev}\times} W^s(Q_\up). 
\end{eqnarray*} 
Here we emphasize that in the moduli spaces $\cM^{A_1}(S_{\op}\,,S_{p_1})$, $\cM^{A_2}(S_{p_1},S_{p_2})$, $\dots$ the $S^1$-action has \emph{not} been quotiented out, as opposed to the previous section where we have quotiented out the $S^1$-action in order to define the $S^1$-equivariant differential. 
It follows from~\cite[Lemma~3.6]{BOauto} that, for a generic choice of 
the collection of Morse functions $\{f_p\}$, the previous fibered 
product is a smooth manifold of dimension  
\begin{eqnarray*}  
\lefteqn{\dim \, \cM^A_m(Q_\op,Q_\up;H,\{f_p\},J,g)} \\ 
& = & 
-\mu(\op)+\ind_{f_\op}(Q_\op) + \mu(\up) - \ind_{f_\up}(Q_\up) + 
2\langle c_1(T\widehat W),A\rangle -1. 
\end{eqnarray*}  
We denote  
$$ 
\cM^A(Q_\op,Q_\up;H,\{f_p\},J,g):=\bigcup_{m\ge 0}  \cM^A_m(Q_\op,Q_\up;H,\{f_p\},J,g). 
$$ 
 
Given a free homotopy class $a$ of loops in $\widehat W$, we define the {\bf 
  parametrized Morse-Bott chain complex} $BC^{a,N}_*(H,\{f_p\},J,g)$  
as a chain complex whose underlying $\Lambda_\omega$-module is  
$$ 
BC^{a,N}_*(H):=BC^{a,N}_*(H,\{f_p\},J,g) := \bigoplus_{S_p\subset \cP^a(H)} \Lambda_\omega \langle 
m_p,M_p\rangle.  
$$ 
The grading is given by  
\begin{eqnarray*} 
|m_p \, e^A| & := & -\mu(\gamma,\lambda) + 1 - 2\langle c_1(T\widehat 
 W),A\rangle, \\ 
|M_p\, e^A| & := & -\mu(\gamma,\lambda) - 2\langle c_1(T\widehat 
 W),A\rangle. 
\end{eqnarray*}  
The {\bf parametrized Morse-Bott differential}  
$$ 
d:BC^{a,N}_*(H)\to BC^{a,N}_{*-1}(H) 
$$  
is defined by 
\begin{equation} \label{eq:dMB}  
dQ_\op := \sum_{\substack{ 
  \up\in \cP^a(H), Q_\up\in \mathrm{Crit}(f_{\up}) \\ 
|Q_\op| - |Q_\up\, e^A|=1}} 
\ \sum_{\scriptstyle \u\in \cM^A(Q_\op,Q_\up;H,\{f_p\},J,g)} 
\epsilon(\u)Q_\up\, e^A, \quad Q_\op \in {\rm Crit}(f_\op). 
\end{equation}  
The sign $\eps(\u)$ is determined by the fibered-sum rule from 
coherent orientations on the relevant spaces of Fredholm operators, as  
explained in~\cite[Section~4.4]{BOauto}.  
 
The Correspondence Theorem~3.7 in~\cite{BOauto} carries over to this situation to show that there is a bijective correspondence preserving signs between the moduli spaces $\cM^A(Q_\op,Q_\up;H,\{f_p\},J,g)$ with $|Q_\op| - |Q_\up\, e^A|=1$ and the moduli spaces $\cM^A(Q_\op,Q_\up;H',J,g)$, where $H'$ is a suitable perturbation of $H$ defined using the collection of Morse functions $\{f_p\}$. Since the differential $\p$ in~\eqref{eq:diff-param} for parametrized contact homology squares to zero, this implies $d^2=0$ as well. 

Similarly to the construction of symplectic homology, we define 
$$ 
BH^{a,N}_*(W):=\lim_{\stackrel \longrightarrow {H\in 
  \cH^{S^1}_{N,\reg}}} H_*(BC^{a,N}_*(H),d), 
$$ 
where the direct limit is taken with respect to increasing homotopies 
of Hamiltonians. It then follows from the Correspondence Theorem~3.7 
in~\cite{BOauto} that  
$$ 
BH^{a,N}_*(W)=SH^{a,S^{2N+1}}_*(W).  
$$ 
 
We now define a filtration on $BC^{a,N}_*(H)$ as follows. Let  
$$ 
B_kC^{a,N}_*(H) := \bigoplus_{ 
\substack{S_p\subset \cP^a(H) \\ A\in H_2(W;\Z) \\ -\mu(p)-2\langle 
  c_1(T\widehat W),A\rangle =k}} \langle m_p\, e^A, \ M_p \, e^A\rangle. 
$$ 
 
\begin{proposition} The $\Z$-modules  
\begin{equation} \label{eq:filtration}
F_\ell BC^{a,N}_*(H)
:= \bigoplus _{k\le \ell} B_kC^{a,N}_*(H), 
\qquad \ell\in\Z 
\end{equation}
form an increasing filtration on 
$BC^{a,N}_*(H)$.  
\end{proposition}  
 
\begin{proof} The formula~\eqref{eq:dMB} involves elements $Q_\op$, 
  $Q_\up$ satisfying $|Q_\op|-|Q_\up\, e^A|=1$, i.e.  
$$ 
-\mu(\op)+\mathrm{ind}_{f_\op}(Q_\op) + \mu(\up) 
-\mathrm{ind}_{f_\up}(Q_\up) + 2\langle c_1(T\widehat W),A\rangle =1. 
$$ 
Since 
$\mathrm{ind}_{f_\op}(Q_\op)-\mathrm{ind}_{f_\up}(Q_\up)\in\{-1,0,1\}$, 
we obtain that 
$$-\mu(\op)+ \mu(\up)+ 2\langle c_1(T\widehat 
W),A\rangle\in\{0,1,2\}.
$$  
\end{proof}  
 
The differential $d$ splits as  
$$ 
d=d^0+d^1+d^2 
$$ 
with $d^r:B_kC^{a,N}_*(H)\to B_{k-r}C^{a,N}_*(H)$. The complex 
$BC^{a,N}_*(H)$ admits a bi-grading which, for an element $Q_p\, e^A$ 
is $(-\mu(p)-2\langle c_1(T\widehat 
W),A\rangle,\mathrm{ind}_{f_p}(Q_p))$. The associated 
spectral sequence $(E^{a,N;r}_{*,*}(H),\bar d^r)$ is supported in two lines 
and converges to $SH^{a,S^{2N+1}}_*(H)$. In particular, it degenerates 
at $r=2$ and takes the form of a long exact sequence~\cite{BOcont}  
\begin{equation}  \label{eq:seq} 
{ 
\xymatrix 
@C=13pt 
@R=10pt@W=1pt@H=1pt 
{ 
\dots \ar[r] & SH^{a,S^{2N+1}}_k(H) \ar[r] 
 & E^{a,N;2}_{k,0}(H) \ar[r]^{\bar d^2} & E^{a,N;2}_{k-2,1}(H) \ar[r] 
& SH^{a,S^{2N+1}}_{k-1}(H) \ar[r] & \dots 
} 
} 
\end{equation}

\begin{proposition}  \label{prop:d0} 
 The differential 
 $d^0:B_kC^{a,N}_*(H)\to B_kC^{a,N}_*(H)$ vanishes. 
\end{proposition}  
 
\begin{proof}
By definition $d^0(Q_p)$, 
  $Q_p\in\mathrm{Crit}(f_p)$ involves critical points of $f_\up$, 
  $\up\in\cP^a(H)$ satisfying $-\mu(p)+\mu(\up) +2\langle 
  c_1(T\widehat W),A\rangle=0$. On the other hand, the dimension of 
  the moduli spaces $\cM^{A_1}(S_{p_1},S_{p_2};H,J,g)$ is equal to 
  $-\mu(p_1)+\mu(p_2)+2\langle c_1(T\widehat W),A_1\rangle$. Since 
  these moduli spaces carry a free $S^1$-action, their dimension must 
  be at least $1$. This proves that $d^0(Q_p)$ counts only gradient 
  trajectories of $f_p$ emanating from $Q_p$. In particular 
  $d^0(M_p)=0$, and $d^0(m_p)$ is either $0$ or equal to $\pm 2 M_p$.  
As explained in Remark~\ref{rmk:coherent} the moduli spaces of $S^1$-equivariant Floer trajectories admit a system of coherent orientations, because the asymptotic operators $D_p$, $p\in\cP(H)$ depend only on $S_p$ when read in $S^1$-invariant trivializations along $S_p$. The arguments of~\cite[Lemma~4.28]{BOauto} then imply that $d^0(m_p)=0$ since the analogue of the twisting operator $T$ used therein is in our case constant for each critical $S^1$-orbit $S_p$. 
\end{proof}  
 
As a consequence, the term $E^{a,N;1}_{*,*}(H)$ can be expressed as  
$$ 
E^{a,N;1}_{*,*}(H) = \bigoplus_{S_p\subset \cP^a(H)} 
\Lambda_\omega \langle m_p,M_p\rangle.  
$$  
Let us denote by $M$ the generator of $H_0(S^1)$ and by $m$ the 
generator of $H_1(S^1)$. It follows from the 
definition~\eqref{eq:SCS1} of the $S^1$-equivariant chain complex that 
there is a natural isomorphism of $\Lambda_\omega$-modules which 
preserves the bi-degree  
$$ 
\Phi:E^{a,N;1}_{*,*}(H) \stackrel \sim \to SC^{a,S^1,N}_*(H) \otimes H_*(S^1), 
$$ 
given by  
$$ 
\Phi(m_p):= S_p\otimes m,\qquad \Phi(M_p):=S_p\otimes M. 
$$ 
 
\begin{proposition} \label{prop:d1} 
There is a commutative diagram  
$$ 
\xymatrix 
@R=25pt 
{E^{a,N;1}_{*,*}(H) \ar[r]^-\Phi \ar[d]_{\bar d^1} & SC^{a,S^1,N}_*(H) 
  \otimes H_*(S^1) \ar[d]^{\partial^{S^1}\otimes \mathrm{Id}} \\ 
E^{a,N;1}_{*,*}(H) \ar[r]_-\Phi & SC^{a,S^1,N}_*(H) 
  \otimes H_*(S^1) 
} 
$$ 
\end{proposition}  
 
\begin{proof} By definition $\bar d^1(Q_\op)$ involves critical points 
  of $f_\up$, $\up\in\cP^a(H)$ such that 
  $-\mu(\op)+\mu(\up)+2\langle c_1(T\widehat W),A\rangle=1$. It 
  follows from the dimension formula~\eqref{eq:dimMS} that 
  $\cM^A(Q_\op,Q_\up;H,\{f_p\},J,g)$ involves exactly one parametrized 
  Floer trajectory $u_1\in\cM^A(S_\op,S_\up;H,J,g)$. Since the 
  dimension of the moduli space $\cM^A(Q_\op,Q_\up;H,\{f_p\},J,g)$ is 
  zero, it follows that either $\oev(u_1)=M_\op$ and $Q_\up=M_\up$, or 
  $\uev(u_1)=m_\up$ and $Q_\op=m_\op$.  
 
  Using that the $S^1$-action on $\cM^A(S_\op,S_\up;H,J,g)$ is free, 
  we see that the coefficient of $Q_\up\, e^A$ in $\bar d^1(Q_\op)$ is 
  given by an algebraic count of connected components of 
  $\cM^A(S_\op,S_\up;H,J,g)$. The latter are in bijective 
  correspondence with elements of 
  $\cM^A_{S^1}(S_\op,S_\up;H,J,g)$, and the signs are the same  
  by our convention for orienting the latter moduli space. Thus, the 
  coefficient of $Q_\up\, e^A$ in $\bar d^1(Q_\op)$ is equal to the 
  coefficient of $S_\up\, e^A$ in $\partial^{S^1}(S_\op)$. This proves 
  the Proposition.   
\end{proof}  
 
\begin{proof}[Proof of Proposition~\ref{prop:partialS1}]  
The claim $\partial^{S^1} \!\!\circ \, \partial^{S^1}=0$ follows directly 
from Proposition~\ref{prop:d1}, using that $\bar d^1 \circ \bar 
d^1=0$. 
\end{proof}  
 
\begin{remark}
The relation $\partial^{S^1} \!\!\circ \, \partial^{S^1}=0$ can also be proved using the usual compactness/gluing argument in Floer homology, the main point being that we are in an $S^1$-invariant, yet transverse, situation. The difference with respect to usual Floer theory is that we are dealing with Morse-Bott asymptotes, for which the relevant analysis has been carried out in~\cite{BOauto}. For compactness, we use that a $2$-dimensional $S^1$-invariant family, which we view as a $1$-dimensional family modulo $S^1$, degenerates into a pair of trajectories with a common asymptote, together with their simultaneous translates by the $S^1$-action. The gluing analysis is also similar to the one in Floer theory, except that it has to be carried out invariantly with respect to the $S^1$-action. 
\end{remark} 
 
It follows from Proposition~\ref{prop:d1} that $\Phi$ induces an 
isomorphism which respects the bi-degree 
\begin{equation} \label{eq:barPhi} 
\bar \Phi : E^{a,N;2}_{*,*}(H) \stackrel \sim \to SH^{a,S^1,N}_*(H) 
  \otimes H_*(S^1). 
\end{equation}

We are now ready to prove Theorem~\ref{thm:SGysin}. We need two  
preparatory Lemmas.  
 
\begin{lemma} \label{lem:limit}  
 We have $\lim_{N\to \infty} SH_*^{S^{2N+1}}(W) = SH_*(W)$ in 
 each degree. 
\end{lemma}  
 
\begin{proof} The limit $\lim_{N\to \infty} SH_*^{S^{2N+1}}(W)$ is 
 taken with respect to the maps $S\iota_*$ corresponding to the 
 inclusions $\iota:S^{2N+1}\hookrightarrow S^{2N+3}$ as in 
 Remark~\ref{rmk:SHincl}. Moreover, we saw that, modulo the K\"unneth 
 isomorphism $SH_*^{S^{2N+1}}(W)\simeq SH_*(W)\otimes H_*(S^{2N+1})$ 
 proved in Proposition~\ref{prop:Kunneth}, the map $S\iota_*$ is equal to 
 $\mathrm{Id}\otimes \iota_*$. The conclusion follows. 
\end{proof}  
 
\begin{lemma} \label{lem:continuation}  
 Let $H_s$ be a smooth increasing homotopy from 
 $H_0\in\cH^{S^1}_{N,\reg}$ to $H_1 \in\cH^{S^1}_{N,\reg}$. Let 
 $(J_i,g_i)\in\cJ^{S^1}_{N,\reg}(H_i)$, $i=0,1$ and $(J_s,g_s)$ a 
 regular smooth homotopy in $\cJ^{S^1}_N$ from $(J_0,g_0)$ to 
 $(J_1,g_1)$. The induced chain morphism $BC^{a,S^1,N}_*(H_0,J_0,g_0)\to 
 BC^{a,S^1,N}_*(H_1,J_1,g_1)$ respects the 
 filtrations~\eqref{eq:filtration}.  
\end{lemma}  
 
The proof of Lemma~\ref{lem:continuation} is given in 
Section~\ref{sec:continuation} below. 
Note that the increasing assumption in the Lemma ensures that the action filtration is respected too.
 
\begin{proof}[Proof of Theorem~\ref{thm:SGysin}] 
 Using the isomorphism $\bar \Phi$ in~\eqref{eq:barPhi}, the long 
 exact sequence~\eqref{eq:seq} becomes 
$$ 
... \to SH^{a,S^{2N+1}}_k\!(H) \to SH^{a,S^1,N}_k\!(H) \to 
SH^{a,S^1,N}_{k-2}\!(H) \to SH^{a,S^{2N+1}}_{k-1}\!(H) \to ... 
$$ 
By Lemma~\ref{lem:continuation}, a smooth increasing homotopy of 
Hamiltonian families in $\cH^{S^1}_{N,\reg}$ induces a filtered chain 
morphism, and therefore a commutative diagram of exact sequences.  
Passing to the direct limit over $H\in\cH^{S^1}_{N,\reg}$ and using 
that the direct limit functor is exact, we obtain a long exact 
sequence  
$$ 
... \to SH^{a,S^{2N+1}}_k\!(W) \to SH^{a,S^1,N}_k\!(W) \to 
SH^{a,S^1,N}_{k-2}\!(W) \to SH^{a,S^{2N+1}}_{k-1}\!(W) \to ... 
$$ 
Passing further to the direct limit over $N\to\infty$, and using 
Lemma~\ref{lem:limit}, we obtain  
$$ 
... \to SH^a_k(W) \to SH^{a,S^1}_k(W) \to 
SH^{a,S^1}_{k-2}(W) \to SH^a_{k-1}(W) \to ... 
$$ 
\end{proof}

\subsection{The Gysin sequence and the cone construction} \label{sec:Gysin-cone} 
 
In this section we relate the Gysin sequence to the cone construction in homological algebra, and then we prove Theorem~\ref{thm:grid}. 

We first recall the definition of the cone of a chain morphism. 
Let $(A_*,\p_A)$ be a homological chain complex and denote $A[k]_*:=(A_{*+k},(-1)^k\p_A)$ for $k\in\Z$.
Given a degree $0$ chain map $f:(A_*,\p_A)\to (A'_*,\p_{A'})$, so that $f\p_A-\p_{A'}f=0$, we 
define the {\bf cone of \boldmath$f$} as the chain complex  
$$ 
\mathcal{C}(f)_*:=A'[1]_*\oplus A_* = A'_{*+1}\oplus A_*, 
$$ 
with differential $\p$ given in matrix form by  
$$ 
\p:=\left(\begin{array}{cc} \p_{A'[1]} & f \\ 0 & \p_A \end{array}\right)=\left(\begin{array}{cc} -\p_{A'} & f \\ 0 & \p_A \end{array}\right). 
$$ 
There is a short exact sequence of complexes  
\begin{equation} \label{eq:short}  
\xymatrix{0\ar[r] & A'[1]_* \ar[r]^i &  \mathcal{C}(f)_* \ar[r]^p & 
  A_* \ar[r] & 0, 
} 
\end{equation}  
with $i$, $p$ the obvious inclusion, respectively projection. The main 
property of the cone construction is that the connecting homomorphism 
in the homology long exact sequence associated to~\eqref{eq:short} is 
precisely $f_*:H_*(A)\to H_*(A')=H_{*-1}(A'[1])$.  
 
\begin{lemma} \label{lem:grid} 
 Let  
\begin{equation} \label{eq:fgh} 
\xymatrix{ 
0\ar[r] & A_* \ar[r] \ar[d]^f & B_* \ar[r] \ar[d]^g & C_* \ar[r] 
\ar[d]^h & 0 \\ 
0\ar[r] & A'_* \ar[r]  & B'_* \ar[r]  & C'_* \ar[r] 
 & 0 
} 
\end{equation}  
be a morphism of short exact sequences of complexes. This induces the 
commutative diagram of homological long exact sequences 
\begin{equation} \label{eq:big-grid} 
\xymatrix 
@C=15pt 
@R=20pt 
{ 
 & \vdots \ar[d] & \vdots \ar[d] & \vdots \ar[d] & \vdots \ar[d] & \\ 
\cdots \ar[r] & H_*(A'[1]) \ar[r] \ar[d] & H_*(B'[1]) \ar[r] \ar[d] 
& H_*(C'[1]) \ar[r] \ar[d] & H_{*-1}(A'[1]) \ar[r] \ar[d] & \cdots \\ 
\cdots \ar[r] & H_*(\mathcal{C}(f)) \ar[r] \ar[d] & 
H_*(\mathcal{C}(g)) \ar[r] \ar[d] & H_*(\mathcal{C}(h)) 
\ar[r] \ar[d] & H_{*-1}(\mathcal{C}(f)) \ar[r] \ar[d] & \cdots \\ 
\cdots \ar[r] & H_*(A) \ar[r] \ar[d]^{f_*} & H_*(B) \ar[r] 
\ar[d]^{g_*} & H_*(C) \ar[r] \ar[d]^{h_*} & H_{*-1}(A) \ar[r] 
\ar[d]^{f_*} & \cdots \\  
\cdots \ar[r] & H_{*-1}(A'[1]) \ar[r] \ar[d] & H_{*-1}(B'[1]) \ar[r] \ar[d] 
& H_{*-1}(C'[1]) \ar[r] \ar[d] & H_{*-2}(A'[1]) \ar[r] \ar[d] & \cdots \\ 
& \vdots & \vdots & \vdots & \vdots &  
} 
\end{equation}  
in which the bottom right square 
$$
\xymatrix
@C=20pt 
@R=20pt 
{
H_*(C)\ar[r]^-{\p_B} \ar[d]^{h_*} & H_{*-1}(A) \ar[d]^{f_*} \\
H_*(C') \ar[r]_-{\p_{B'[1]}} & H_{*-1}(A')
}
$$
anti-commutes, i.e. $f_*\p_B+\p_{B'[1]}h_*=0$.
\end{lemma}

\begin{proof} 
Applying the cone construction to each column of~\eqref{eq:fgh} we 
obtain the short exact sequence of short exact sequences of complexes  
\begin{equation} \label{eq:cone-fgh} 
\xymatrix 
@C=25pt 
@R=15pt 
{ 
& 0 \ar[d] & 0 \ar[d] & 0 \ar[d] & \\ 
0 \ar[r] & A'[1]_* \ar[d] \ar[r] & B'[1]_* \ar[d] \ar[r] & C'[1]_* 
\ar[d] \ar[r] & 0 \\ 
0 \ar[r] & \mathcal{C}(f)_* \ar[d] \ar[r] & \mathcal{C}(g)_* \ar[d] 
\ar[r] & \mathcal{C}(h)_* \ar[d] \ar[r] & 0 \\ 
0 \ar[r] & A_* \ar[d] \ar[r] & B_* \ar[d] \ar[r] & C_* 
\ar[d] \ar[r] & 0 \\ 
& 0 & 0 & 0 &  
} 
\end{equation}  
The lines/columns in~\eqref{eq:big-grid} are obtained as homological long exact 
sequences associated to the horizontal/vertical short exact sequences  
in~\eqref{eq:cone-fgh}. Commutativity of the horizontal strips 
in~\eqref{eq:big-grid} follows from functoriality of the homological 
long exact sequence with respect to morphisms of short exact sequences 
of complexes. More precisely, for the first two horizontal strips 
in~\eqref{eq:big-grid} we use~\eqref{eq:cone-fgh}, and for the third 
horizontal strip in~\eqref{eq:big-grid} we use~\eqref{eq:fgh}. That the square under consideration is anti-commutative is a consequence of the identity $\p_{B'[1]}=-\p_{B'}$.
\end{proof}

In the next Lemma we interpret the homology long exact sequence of a cone as a Gysin long exact sequence arising from a spectral sequence supported on two lines. This fact is certainly folk knowledge, but we were unable to find a suitable reference. One can view it as an algebraic reformulation of Thom's interpretation of the Gysin sequence in~\cite{Thom} (see also~\cite[p.~1192]{Malm}). 

Let $(A_*,\p_A)$, $(A'_*,\p_{A'})$ be homological chain complexes and $f:A_* \to A'[-2]_*$ a degree $0$ chain map, i.e. $f\p_A=\p_{A'[-2]}f$. (Thus $f:A_*\to A'_{*-2}$ is a degree $-2$ chain map such that $f\p_A=\p_{A'}f$.) 
We consider the following two algebraic constructions.

(i) Let $\cC(f):=(A'[-1]\oplus A,\p)$ be the cone of $f$. The short exact sequence~\eqref{eq:short} 
$$
0\to A'[-1]_*\stackrel i \longrightarrow A'[-1]_*\oplus A_*\stackrel p \longrightarrow A_*\to 0
$$
induces the {\bf homological long exact sequence of the cone}
\begin{equation} \label{eq:alg-cone}
\dots H_{*-1}(A')\stackrel {i_*}\to H_*(\cC(f))\stackrel {p_*}\to H_*(A)\stackrel {f_*}\to H_{*-2}(A')\stackrel{i_*}\to H_{*-1}(\cC(f))\dots,
\end{equation}
with $i$, $p$ being the inclusion, resp. projection.
 
(ii) Let $C_{*,*}(f):=\{(C_{p,q},d), \, p,q\ge 0\}$ be the first quadrant double complex supported on the lines $q=0,1$ and defined by
$$
C_{*,0}:=A_*, \qquad C_{*,1}:=A'_*, \qquad d:=d^1+d^2, 
$$
with $d^r:C_{p,q}\to C_{p-r,q+r-1}$, $r=1,2$ and 
$$
\qquad d^1|_{C_{*,0}}=\p_A,\qquad d^1|_{C_{*,1}}:=-\p_{A'},\qquad d^2|_{C_{*,0}}:=f.
$$
In particular $(d^1)^2=0$ and $d^1d^2+d^2d^1=0$.

$$
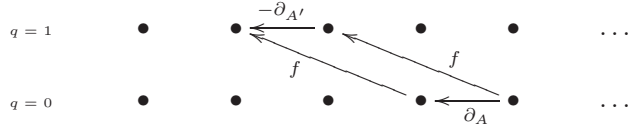
\xymatrix
@R=5pt { 
\mbox{\tiny$q=1$}  &  \bullet  &  \bullet   &  \bullet \ar[l]_{-\p_{A'}} &
\bullet  &  \bullet & \ldots & \\ 
\qquad \\ 
\mbox{\tiny$q=0$}  &  \bullet  &  \bullet   &  \bullet &  \bullet \ar[uull]^(.65){f}
&  \bullet \ar[uull]_(.35){f} \ar[l]^(.40){\p_A} & \ldots & \\ 
\qquad 
}
$$
We consider on the total complex $C^{tot}(f)_*$ the filtration 
$$
F_pC^{tot}_*:=\oplus_{\ell\le p} \oplus_q C_{\ell,q}
$$
and denote the associated spectral sequence by $(E^r_{p,q},\bar d^r)$, $r\ge 0$. This converges to the homology $H_*:=H_*(C^{tot}(f))$ of the total complex and it degenerates at the third page for dimensional reasons, yielding the exact sequences 
$$
0\to E^3_{p,0}\to E^2_{p,0}\stackrel{\bar d^2}\longrightarrow E^2_{p-2,1}\to E^3_{p-2,1}\to 0
$$
and 
$$
0\to E^3_{p-1,1}\to H_p\to E^3_{p,0}\to 0.
$$
These can be assembled according to the diagram
$$
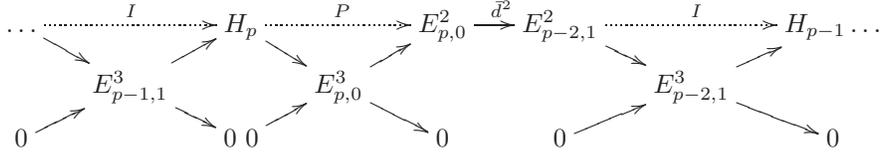
\xymatrix
@R=5pt
@C=15pt{
\dots \ar[dr] \ar@{.>}[rr]^I & & H_p \ar[dr] \ar@{.>}[rr]^P & & E^2_{p,0} \ar[r]^{\bar d^2} & E^2_{p-2,1} \ar[dr] \ar@{.>}[rr]^I & & H_{p-1} \dots \\
& E^3_{p-1,1} \ar[ur] \ar[dr] & & E^3_{p,0} \ar[ur] \ar[dr] & & & E^3_{p-2,1} \ar[ur] \ar[dr] & & \\
0 \ar[ur] & & 0 \ 0 \ar[ur] & & 0 & 0\ar[ur] & & 0 & 
}
$$
Here we denote by $I$ the composition $E^2_{p-1,1}\to E^3_{p-1,1}\to H_p$, and by $P$ the composition $H_p\to E^3_{p,0}\to E^2_{p,0}$. Taking into account that $E^2_{p,0}=H_p(A)$ and $E^2_{p,1}=H_p(A')$, the top line translates into the long exact sequence
\begin{equation} \label{eq:alg-Gysin}
\dots H_{p-1}(A')\stackrel{I}\to H_p(C^{tot}(f))\stackrel{P}\to H_p(A)\stackrel{\bar d^2}\to  H_{p-2}(A')\stackrel{I}\to H_{p-1}(C^{tot}(f))\dots
\end{equation}
We call~\eqref{eq:alg-Gysin} the {\bf Gysin exact sequence} of the complex $C_{*,*}(f)$.

\begin{lemma} \label{lem:Gysin}
The homological long exact sequence of the cone $\cC(f)$ and the Gysin exact sequence of the complex $C_{*,*}(f)$ coincide. 
\end{lemma}
  
\begin{proof}
It follows from the definitions that the total complex $C^{tot}:=C^{tot}(f)$ coincides with the cone $\cC(f)$. We need to show $I=i_*$, $P=p_*$, $\bar d^2=f_*$.

That $\bar d^2=f_*$ is a consequence of the fact that the differential $d$ has no component $d^0:C_{p,1}\to C_{p,0}$. Then $d^2$ is a chain map on $E^1_{*,*}=C_{*,*}$ and $\bar d^2$ is induced by $d^2=f$. 

In order to prove the identities $P=p_*$ and $I=i_*$ we consider the following two additional filtrations on $C^{tot}$: 
$$
F'_pC^{tot}_*:=F_{p-1}C^{tot}_*,\qquad \qquad F''_pC^{tot}_*:=\oplus_{\ell\le p} C^{tot}_\ell.
$$
[The second filtration is also known as the \emph{filtration b\^ete}, or \emph{tautological filtration}.] We denote $'C^{tot}_*$ and $''C^{tot}_*$ the complex $C^{tot}_*$ endowed with the filtrations $F'$, respectively $F''$. Also, we think of $C^{tot}_*$ as carrying the filtration $F$. Then the identity maps $\mathrm{Id}:{'C^{tot}}\to {''C^{tot}}$ and $\mathrm{Id}:{''C^{tot}}\to C^{tot}$ are chain maps that respect the filtrations. Similarly, we endow $A'[-1]_*$ and $A_*$ with the tautological filtrations, with respect to which the inclusion $A'[-1]\to {'C^{tot}}$ and the projection $C^{tot}\to A$ are chain maps that respect the filtrations. To summarize, we have the following sequence of filtered complexes and chain maps 
\begin{equation}
\xymatrix{
A'[-1] \ar[r]^i & {'C^{tot}} \ar[r]^{\mathrm{Id}} & {''C^{tot}} \ar[r]^{\mathrm{Id}} & C^{tot} \ar[r]^p & A
}
\end{equation}
In particular, these maps induce morphisms between the corresponding spectral sequences. The two main observations are now the following: 
\begin{itemize}
\item the composition $\mathrm{Id}\circ i:A'[-1]\to {''C^{tot}}$ induces on the $3^{\mbox{\tiny rd}}$ page the map $I$. 
\item the composition $p\circ\mathrm{Id}:{''C^{tot}}\to A$ induces on the $3^{\mbox{\tiny rd}}$ page the map $P$.
\end{itemize}
Indeed, the spectral sequences for the tautological filtrations are supported on the single line $q=0$, whereas the spectral sequence for $'C^{tot}$ is supported on the lines $q=-1$ and $q=0$, with $'E^r_{p,0}=E^r_{p-1,1}$ and $'E^r_{p,-1}=E^r_{p-1,0}$. Thus the composition $\mathrm{Id}\circ i$ is equal on the third page to $H_{p-1}(A')\to E^3_{p-1,1}\to H_p$, which is $I$. Also, the composition $p\circ\mathrm{Id}$ is equal on the third page to $H_p\to E^3_{p,0}\to H_p(A)$, which is $P$.

On the other hand, the map induced on the third page by $\mathrm{Id}\circ i=i$ is $i_*$ since both the source and the target carry the tautological filtrations. Similarly, the map induced on the third page by $p\circ\mathrm{Id}=p$ is $p_*$. This completes the proof. 
\end{proof}

\begin{remark}
In the previous proof, the identities $P=p_*$ and $I=i_*$ can also be checked directly from the definition of the spectral sequence, by explicitly computing the filtration $F_pH_*$ induced on $H_*$ by the filtration $F_pC^{tot}_*$. 
\end{remark}

\begin{proof}[Proof of Theorem~\ref{thm:grid}]  
Given $H\in\cH^{\prime\, S^1}_{N,\reg}$, $(J,g)\in 
\cJ^{S^1}_{N,\reg}(H)$, and a generic collection of perfect Morse 
functions $f_p:S_p\to \R$, $p\in\cP^0(H)$, we denote 
$C_*:=BC^N_*(H,\{f_p\},J,g)$ (we recall that we work in the 
trivial free homotopy class). Filtering by the action as in the 
definition of $SC^{\pm,S^1,N}_*(H,J,g)$ in 
Section~\ref{sec:S1equivsymplhom}, we obtain filtered complexes  
$C^\pm_*:=BC^{\pm,N}_*(H,\{f_p\},J,g)$. We denote by 
$(E^{\pm,N;r}_{*,*},\bar d^r)$ and $(E^{N;r}_{*,*},\bar d^r)$ the corresponding spectral sequences, 
which degenerate at $r=3$ for dimensional reasons.  
 
Since $d^0=0$ we have $E^{N;1}_{*,*}=E^{N;0}_{*,*}$, the differential $\bar d^1$ is canonically identified with $d^1$, and the differential $\bar d^2$ on $E^{N;2}_{*,*}$ is induced by $d^2$ viewed as a chain map on $(E^{N;1}_{*,*}, d^1)$. The same holds for $E^{\pm,N;*}_{*,*}$. We thus have
a short exact sequence 
$0\to (E^{-,N;1}_{*,*},d) \to 
(E^{N;1}_{*,*},d) \to (E^{+,N;1}_{*,*},d) \to 0$ with $d= 
d^1+d^2$. This can be rewritten as a morphism of short exact 
sequences of chain complexes 
\begin{equation} \label{eq:bard2}  
\xymatrix{ 
0 \ar[r] & (E^{-,N;1}_{*,0},d^1) \ar[r] \ar[d]^{d^2} & 
(E^{N;1}_{*,0},d^1) 
\ar[r] \ar[d]^{d^2} & (E^{+,N;1}_{*,0},d^1) \ar[r] \ar[d]^{d^2} & 0 
\\ 
0 \ar[r] & (E^{-,N;1}_{*-2,1},-d^1) \ar[r] & (E^{N;1}_{*-2,1},-d^1) 
\ar[r] & (E^{+,N;1}_{*-2,1},-d^1) \ar[r] & 0 
} 
\end{equation}   
 
We claim that the commutative diagram~\eqref{eq:grid} in the statement 
is obtained by applying Lemma~\ref{lem:grid} to~\eqref{eq:bard2}.  
This follows from the following three observations. Firstly, the cone 
$\mathcal{C}(d^2)$ is canonically identified with 
$(C,d)$, respectively 
$(C^\pm,d)$, so that its homology is 
$H_*(BC^N_*(H),d)$, resp. $H_*(BC^{\pm,N}_*(H),d)$. Secondly, 
the homology of $(E^{N;1}_{*,i},d^1)$, $i=0,1$ is isomorphic to 
$SH^{S^1,N}_*(H)$, and the homology of $(E^{\pm, N;1}_{*,i},d^1)$, $i=0,1$ is  
isomorphic to $SH^{\pm,S^1,N}_*(H)$. Thirdly, 
Lemma~\ref{lem:Gysin} 
shows that, via the above identifications with the cone 
$\mathcal{C}(d^2)$, the Gysin exact sequences obtained from the  
spectral sequences $E^{N;r}_{*,*}$ and $E^{\pm,N;r}_{*,*}$ coincide 
with the homological long exact sequences of the corresponding cone 
constructions.  
 
Passing to the direct limit on $H\in\cH^{\prime\, S^1}_{N,\reg}$ and 
$N\to\infty$ we obtain the commutative diagram~\eqref{eq:grid}.  
\end{proof}  
 
\begin{remark} \label{rmk:Delta} Denoting the maps in the Gysin exact sequence by 
$$
\xymatrix
@C=20pt
{
\dots \ar[r] & SH_*(W) \ar[r]^E & SH_*^{S^1}(W) \ar[r]^D &
SH_{*-2}^{S^1}(W) \ar[r]^M & SH_{*-1}(W) \ar[r] & \dots
}
$$
we defined in the Introduction the Batalin-Vilkovisky (BV) operator 
$$
\Delta:=M\circ E:SH_*(W)\to SH_{*+1}(W).
$$
The 
interpretation given by Lemma~\ref{lem:Gysin}
of the Gysin exact sequence as the long exact sequence of the cone $C_*:=\cC(d^2)$ allows us to give the following description of $\Delta$ at chain level. We identify $C_*=BC_*^N(H,\{f_p\},J,g)$ with $SC_{*-1}^{S^1}\oplus SC_*^{S^1}:=SC_{*-1}^{S^1,N}(H,J,g)\oplus SC_*^{S^1,N}(H,J,g)$ via
$$
m_p\longmapsto (S_p,0), \qquad M_p\longmapsto (0,S_p).
$$
Via this identification, the map $\Delta$ is induced by the chain map $\bar\Delta :C_*\to C_{*+1}$ given by 
$$
\bar \Delta : SC_{*-1}^{S^1}\oplus SC_*^{S^1}\longrightarrow SC_*^{S^1}\oplus SC_{*+1}^{S^1},
$$
$$
(S_p,S_q) \longmapsto (S_q,0).
$$
Indeed, the short exact sequence of the cone $\cC(d^2)=SC_{*-1}^{S^1}\oplus SC_*^{S^1}$ writes
$$
0\to SC_{*-1}^{S^1}\stackrel i \to SC_{*-1}^{S^1}\oplus SC_*^{S^1}\stackrel p \to SC_*^{S^1}\to 0.
$$
The maps $i$ and $p$ are the canonical inclusion and projection. The connecting homomorphism in the homological long exact sequence is the map $D$, so that the maps 
$i$ and $p$ induce $M$ and $E$ respectively. Hence the composition $\bar\Delta=i\circ p$ induces $\Delta=M\circ E$. We refer to~\cite{BO4} for a description of the BV-operator from a different perspective. 

\end{remark}

\subsection{Filtered continuation maps for parametrized symplectic homology} \label{sec:continuation} 
 
Let $H_s$, $s\in\R$  be a smooth increasing homotopy from $H_-\in\cH^{S^1}_{N,\mathrm{reg}}$ to $H_+\in\cH^{S^1}_{N,\mathrm{reg}}$, such that $H_s\equiv H_-$ for $s<<0$ and $H_s\equiv H_+$ for $s>>0$. Let $(J_\pm,g_\pm)\in\cJ^{S^1}_{N,\reg}(H_\pm)$ and $(J_s,g_s)$, $s\in\R$ a 
 regular smooth homotopy in $\cJ^{S^1}_N$ from $(J_-,g_-)$ to 
 $(J_+,g_+)$, which is constant near $\pm\infty$. Given
$\op\in\cP(H_-)$ and $\up\in\cP(H_+)$, we define the {\bf moduli space
of \boldmath$s$-dependent \boldmath$S^1$-equivariant Floer
trajectories} $\cM(S_\op,S_\up;H_s,J_s,g_s)$ to consist of pairs
$(u,\lambda)$ with    
 $$ 
 u:\R\times S^1\to \widehat W, \qquad \lambda:\R\to S^{2N+1} 
 $$ 
satisfying  
\begin{equation} \label{eq:Floer-u-s} 
\p_s u + J^\theta_{s,\lambda(s)}\p_\theta u - J^\theta_{s,\lambda(s)}X^\theta_{H_{s,\lambda(s)}}(u)=0, 
\end{equation} 
\begin{equation} \label{eq:Floer-lambda-s} 
\dot \lambda(s) -\int_{S^1} \vec\nabla_\lambda H_s(\theta,u(s,\theta),\lambda(s))d\theta=0, 
\end{equation} 
and  
\begin{equation} \label{eq:Floer-asy-s} 
\lim_{s\to-\infty} (u(s,\cdot),\lambda(s))\in S_{\op}, \qquad \lim_{s\to+\infty} (u(s,\cdot),\lambda(s))\in S_{\up}. 
\end{equation} 
Due to the $s$-dependence, the additive group $\R$ does not act on the moduli space $\cM(S_\op,S_\up;H_s,J_s,g_s)$.  
Recall that, for each $p=(\gamma,\lambda)\in\cP(H_\pm)$, we have chosen a 
cylinder $\sigma_p:[0,1]\times S^1\to\widehat W$ such that 
$\sigma_p(0,\cdot)=l_{[\gamma]}$ and $\sigma_p(1,\cdot)=\gamma$. We 
define $\overline \sigma_p(s,\theta):=\sigma_p(1-s,\theta)$. We define  
$$ 
\cM^A(S_\op,S_\up;H_s,J_s,g_s)\subset \cM(S_\op,S_\up;H_s,J_s,g_s) 
$$ 
to consist of trajectories $(u,\lambda)$ such that 
$[\sigma_{\op}\#u\#\overline \sigma_{\up}]=A\in H_2(\widehat 
W;\Z)$. It follows from Proposition~\ref{prop:indexMB} that  
\begin{equation}  \label{eq:dim-MB-s} 
\dim\, \cM^A(S_\op,S_\up;H_s,J_s,g_s) = -\mu(\op) +\mu(\up) +2\langle 
c_1(T\widehat W),A\rangle+1. 
\end{equation} 
 
For each $S^1$-orbit of critical 
points $S_p\subset \cP(H_\pm)$ we choose a perfect Morse function 
$f^\pm_p:S_p\to\R$. We denote by $m_p$, $M_p$ the minimum, respectively 
the maximum of $f^\pm_p$. Given $\op\in\cP(H_-)$, $\up\in\cP(H_+)$, 
$Q_\op\in\mathrm{Crit}(f^-_\op)$, $Q_\up\in\mathrm{Crit}(f^+_\up)$, $A\in H_2(W;\Z)$, and 
$m_\pm\ge 0$, we denote by   
$$ 
\cM^A_{m_-,m_+}(Q_\op,Q_\up;H_s,\{f^\pm_p\},J_s,g_s) 
$$ 
the union for $p^-_1,\dots,p^-_{m_-}\in\cP(H_-)$, $p^+_1,\dots,p^+_{m_+}\in\cP(H_+)$, and $A^-_1+\dots+A^-_{m_-}+A^0+A^+_1+\dots+A^+_{m_+}=A$ of 
the fibered products  
\begin{eqnarray*} 
&& 
W^u(Q_\op)  
\times_{\oev} 
(\cM^{A^-_1}(S_{\op}\,,S_{p^-_1};H_-,J_-,g_-)\!\times\!\R^+)\\ 
&& {_{\varphi_{f^-_{p^-_1}}\!\circ\uev}}\!\times   
_{\oev}\dots {_{\varphi_{f^-_{p^-_{m_- -1}}}\!\circ\uev}}\!\times   
_{\oev} (\cM^{A^-_{m_-}}(S_{p^-_{m_- -1}},S_{p^-_{m_-}};H_-,J_-,g_-)\!\times\!\R^+) \\ 
&& {_{\varphi_{f^-_{p^-_{m_-}}}\!\circ\uev}\times_{\oev}}  
(\cM^{A^0}(S_{p^-_{m_-}},S_{p^+_1};H_s,J_s,g_s)\!\times\!\R^+)\\ 
&& {_{\varphi_{f^+_{p^+_1}}\!\circ\uev}}\!\times   
_{\oev} (\cM^{A^+_1}(S_{p^+_1},S_{p^+_2};H_+,J_+,g_+)\!\times\!\R^+) \\ 
&&{_{\varphi_{f^+_{p^+_2}}\!\circ\uev}}\!\times   
_{\oev}\dots  
{_{\varphi_{f^+_{p^+_{m_+}}}\!\!\circ\uev}}\!\!\times 
_{\oev}   
\cM^{A^+_{m_+}}(S_{p^+_{m_+}},\!S_{\up})  
{_{\uev}\times} W^s(Q_\up). 
\end{eqnarray*} 
It follows from~\cite[Lemma~3.6]{BOauto} that, for a generic choice of 
the collection of Morse functions $\{f^\pm_p\}$, the previous fibered 
product is a smooth manifold of dimension  
\begin{eqnarray*}  
\lefteqn{\dim \, \cM^A_{m_-,m_+}(Q_\op,Q_\up;H_s,\{f^\pm_p\},J_s,g_s)} \\ 
& = & 
-\mu(\op)+\ind_{f^-_\op}(Q_\op) + \mu(\up) - \ind_{f^+_\up}(Q_\up) + 
2\langle c_1(T\widehat W),A\rangle. 
\end{eqnarray*}  
We denote  
$$ 
\cM^A(Q_\op,Q_\up;H_s,\{f^\pm_p\},J_s,g_s):=\bigcup_{m_\pm\ge 0}  \cM^A_{m_-,m_+}(Q_\op,Q_\up;H_s,\{f^\pm_p\},J_s,g_s). 
$$ 
Whenever $\dim\, \cM^A(Q_\op,Q_\up;H_s,\{f^\pm_p\},J_s,g_s)=0$, we can associate a sign $\eps(\u)$ to each of its elements via the choice of coherent orientations and the fibered sum rule~\cite[Section~4.4]{BOauto}. We define the {\bf continuation morphism}  
$$ 
\sigma_{H_+,H_-}:BC_*^{a,N}(H_-)\to BC_*^{a,N}(H_+) 
$$ 
by  
$$ 
\sigma_{H_+,H_-}(Q_\op):=\sum_{\substack{ 
  \up\in \cP^a(H_+), Q_\up\in \mathrm{Crit}(f^+_{\up}) \\ 
|Q_\op| - |Q_\up\, e^A|=0}} 
\ \sum_{\scriptstyle \u\in \cM^A(Q_\op,Q_\up;H_s,\{f^\pm_\gamma\},J_s,g_s)} 
\epsilon(\u)Q_\up\, e^A, 
$$ 
for all $\op\in\cP(H_-)$ and $Q_\op \in {\rm Crit}(f^-_\op)$. In order to emphasize the homotopy used to  
define $\sigma_{H_+,H_-}$, we shall sometimes write $\sigma_{H_+,H_-}^{H_s}$. 
 
\begin{proof}[Proof of Lemma~\ref{lem:continuation}] 
That the map $\sigma_{H_+,H_-}$ is a chain morphism satisfying $\sigma_{H_+,H_-}\circ d=d\circ \sigma_{H_+,H_-}$ follows from a straightforward generalization of the Correspondence Theorem~3.7 in~\cite{BOauto}.  Via the identification of the parametrized Morse-Bott complexes with the Floer complexes of suitable perturbations of the Hamiltonians $H_\pm$, the morphism $\sigma_{H_+,H_-}$ corresponds to the continuation morphism induced by an increasing homotopy of Hamiltonians.  
 
That $\sigma_{H_+,H_-}$ preserves the filtration follows from the fact that each of the moduli spaces  
$\cM^{A^0}(S_{p^-_{m_-}},S_{p^+_1};H_s,J_s,g_s)$, $\cM^{A^-_j}(S_{p^-_{i-1}},S_{p^-_i};H_-,J_-,g_-)$, $1\le i\le m_-$ and  
$\cM^{A^+_i}(S_{p^+_i},S_{p^+_{i+1}};H_+,J_+,g_+)$, $1\le i \le m_+$ carries a free $S^1$-action (we denote $p^-_0=\op$, $p^+_{m_+ +1}=\up$). In case they are nonempty, their dimension is therefore at least $1$. It then follows from the dimension formulas~\eqref{eq:dim-MB-s}  
and~\eqref{eq:dimMS} that $|\op|-|\up e^A|=-\mu(\op)+\mu(\up)+2\langle c_1(T\widehat W),A\rangle\ge 0$.  
\end{proof}  
 
For the next statement it is useful to introduce the following
algebraic concept. Let $(C_*,d_C)$ and $(D_*,d_D)$ be differential
complexes endowed with increasing filtrations $F_\ell C_*$, $F_\ell
D_*$, $\ell\in \Z$. A map   
$K:C_*\to D_*$ is said to be {\bf of order \boldmath$k\ge 0$} if
$K(F_\ell C_*)\subset F_{\ell+k}D_*$ (we allow $K$ to shift the
grading). This definition is relevant in the following context. Assume
$f,g:C_*\to D_*$ are filtration preserving chain maps such that
$f-g=d_D \circ K + K\circ d_C$ for a chain homotopy $K:C_*\to D_{*+1}$
of order $k\ge 0$. Then the maps $f_r, g_r$, $r\ge 0$ induced on the
associated spectral sequences are homotopic for $r=k$, and coincide
for $r>k$~\cite[Exercise~3.8, p.87]{McC}.   
 
\begin{proposition} \label{prop:chain-homotopy} 
Let $H_-\le H_+$ be Hamiltonians in $\cH^{S^1}_{N,\mathrm{reg}}$ and $H_s^0, H_s^1\in \cH^{S^1}_N$, $s\in\R$ be generic smooth increasing homotopies from $H_-$ to $H_+$, which are constant near $\pm\infty$. Let $(J_\pm,g_\pm)\in\cJ^{S^1}_{N,\reg}(H_\pm)$ and $(J^0_s,g^0_s)$, $(J^1_s,g^1_s)$ be two generic  smooth homotopies in $\cJ^{S^1}_N$ from $(J_-,g_-)$ to 
 $(J_+,g_+)$, which are constant near $\pm\infty$. A generic homotopy of homotopies $(H_s^\rho,J^\rho_s,g^\rho_s)$, $\rho\in [0,1]$ induces a map $K:BC_*^{a,N}(H_-)\to BC_{*+1}^{a,N}(H_+)$ of order $1$ such that  
$$ 
\sigma_{H_+,H_-}^{H_s^1} - \sigma_{H_+,H_-}^{H_s^0} = d\circ K+K\circ d. 
$$  
\end{proposition} 
 
\begin{proof} Given $\op\in\cP(H_-)$, $\up\in \cP(H_+)$, $Q_\op\in\mathrm{Crit}(f^-_\op)$, $Q_\up\in\mathrm{Crit}(f^+_\up)$, and $A\in H_2(W;\Z)$ such that  
$$ 
|Q_\op|-|Q_\up e^A|=0,  
$$ 
we define  
$$ 
\cM^A:=\bigcup_{\rho\in [0,1]} \cM^A(Q_\op,Q_\up;H^\rho_s, \{f^\pm_p\},J^\rho_s, g^\rho_s). 
$$ 
For a generic choice of the triple $(H_s^\rho,J^\rho_s,g^\rho_s)$, $\rho\in [0,1]$, the space $\cM^A$ is a smooth $1$-dimensional manifold. Its boundary splits as  
$$ 
\p \cM^A= \p^0\cM^A \cup \p^1\cM^A \cup \p^{int}\cM^A. 
$$ 
Here $\p^i\cM^A=\cM^A(Q_\op,Q_\up;H^i_s, \{f^\pm_p\},J^i_s, g^i_s)$, $i=0,1$ and $\p^{int}\cM^A$ corresponds to degeneracies at some point $\rho\in]0,1[$, namely  
\begin{eqnarray*} 
\lefteqn{\p^{int}\cM^A}\\ 
 &\hspace{-.3cm} = & \!\!\bigcup_{\rho\in ]0,1[} \!\!\cM^B(Q_\op,Q_{p_-};H_-,\{f^-_p\},J_-,g_-) \times \cM^{A-B}(Q_{p_-},Q_\up;H^\rho_s, \{f^\pm_p\},J^\rho_s, g^\rho_s) \\ 
& \hspace{-.3cm}& \hspace{-.5cm}\cup \bigcup_{\rho\in ]0,1[}\!\!\cM^{A-B}(Q_\op,Q_{p_+};H^\rho_s, \{f^\pm_p\},J^\rho_s, g^\rho_s) \times \cM^B(Q_{p_+},Q_\up;H_+,\{f^+_p\},J_+,g_+). 
\end{eqnarray*} 
Here the union is taken over $B\in H_2(W;\Z)$, $p_\pm\in\cP(H_\pm)$, $Q_{p_\pm}\in\mathrm{Crit}(f^\pm_{p_\pm})$ such that $|Q_{p_-}|-|Q_\up e^{A-B}|=-1$ and $|Q_\op|-|Q_{p_+} e^{A-B}|=-1$. For a generic choice of the triple $(H_s^\rho,J^\rho_s,g^\rho_s)$, $\rho\in [0,1]$, there are only a finite number of values of $\rho$ involved in the above union. The elements of $\p^{int}\cM^A$ correspond to the breaking of a gradient trajectory involved in one of the fiber products defining $\cM^A(Q_\op,Q_\up;H^\rho_s, \{f^\pm_p\},J^\rho_s, g^\rho_s)$, as $\rho$ converges to some $\rho_0\in]0,1[$. There are yet two other types of degeneracy in $\cM^A$, which compensate each other: the length of a gradient trajectory in a fibered product as above can shrink to zero, or a Floer trajectory can break at a point $Q\in S_p\setminus \mathrm{Crit}(f^\pm_p)$, for some $p\in\cP(H_\pm)$.  
 
We define $K:BC_*^{a,N}(H_-)\to BC_{*+1}^{a,N}(H_+)$ by 
$$ 
K(Q_\op)=\sum_{\rho\in]0,1[} \ \sum_{|Q_\op|-|Q_\up e^A|=-1} \ \sum_{\u\in\cM^A(Q_\op,Q_\up;H^\rho_s, \{f^\pm_p\},J^\rho_s, g^\rho_s)} \eps(\u)Q_\up e^A. 
$$ 
The above description of $\p\cM^A$ shows that we have indeed $\sigma_{H_+,H_-}^{H_s^1} - \sigma_{H_+,H_-}^{H_s^0} = d\circ K+K\circ d$. That the chain homotopy $K$ is of order $1$ means that it satisfies $K(F_\ell B_*^{a,N}(H_-))\subset F_{\ell+1} B_{*+1}^{a,N}(H_+)$. This follows from the fact that each family of moduli spaces $\bigcup_{\rho\in]0,1[} \cM^{A^0}(Q_{p_-},Q_{p_+};H^\rho_s, \{f^\pm_p\},J^\rho_s, g^\rho_s)$  
carries a free $S^1$-action. In case it is nonempty, its dimension must therefore be at least $1$. On the other hand, it follows from~\eqref{eq:dim-MB-s} that this dimension is equal to $|p_-|-|p_+ e^{A^0}|+2$, so that $|p_-|-|p_+ e^{A^0}|\ge -1$. A similar argument shows that, for the moduli spaces $\cM^{A^\pm}(Q_{p^\pm_0},Q_{p^\pm_1};H_\pm,\{f^\pm_p\},J_\pm,g_\pm)$ appearing in the definition of $K$, we must have $|p^\pm_0|-|p^\pm_1 e^{A^\pm}|\ge 0$. Thus, for the fibered products appearing in the definition of $K$ we have $|\op|-|\up e^A|\ge -1$.  
\end{proof} 
 
\begin{proposition} \label{prop:composition} 
Let $H_0\le H_1\le H_2$ be three Hamiltonians in $\cH^{S^1}_{N,\mathrm{reg}}$, and let $H_s^{01}, H_s^{12}\in  \cH^{S^1}_N$, $s\in\R$ be two generic smooth increasing homotopies from $H_0$ to $H_1$, respectively from $H_1$ to $H_2$, which are constant near $\pm\infty$. Let $(J_i,g_i)\in\cJ^{S^1}_{N,\reg}(H_i)$, $i=0,1,2$ and $(J^{01}_s,g^{01}_s)$, $(J^{12}_s,g^{12}_s)$ be two generic  smooth homotopies in $\cJ^{S^1}_N$ from $(J_0,g_0)$ to 
 $(J_1,g_1)$, respectively from $(J_1,g_1)$ to $(J_2,g_2)$, which are constant near $\pm\infty$. 
For $R>0$ sufficiently large we denote  
$$ 
H^{02,R}_s:=\left\{\begin{array}{ll} 
H^{01}_{s+R}, & s\le 0, \\ 
H^{12}_{s-R}, & s\ge 0. 
\end{array}\right. 
$$ 
We define the homotopies $J^{02,R}_s, g^{02,R}_s$ in a similar way. There exists a map $K:BC_*^{a,N}(H_0)\to BC_{*+1}^{a,N}(H_2)$ of order $1$ such that  
$$ 
\sigma_{H_2,H_1}^{H_s^{12}}\circ \sigma_{H_1,H_0}^{H_s^{01}} - \sigma_{H_2,H_0}^{H^{02,R}_s} = d\circ K+K\circ d. 
$$  
\end{proposition} 
 
\begin{proof} Let $\{f^i_p\}$, $i=0,1,2$ be three generic collections of perfect Morse functions on $S_p$, for $p\in\cP(H_i)$ respectively. Given $\op\in\cP(H_0)$, $\up\in\cP(H_2)$, $Q_\op\in\mathrm{Crit}(f^0_\op)$, $Q_\up\in\mathrm{Crit}(f^2_\up)$, $A\in H_2(W;\Z)$ such that $|Q_\op|-|Q_\up e^A|=0$, we define for $R_0>0$ sufficiently large the family of moduli spaces 
$$ 
\cM^A_1:=\bigcup_{R\ge R_0} \cM^A(Q_\op,Q_\up;H^{02,R}_s, \{f^0_p,f^2_p\},J^{02,R}_s,g^{02,R}_s). 
$$ 
For a generic choice of the homotopies, this is a smooth $1$-dimensional manifold. Its boundary splits as 
$$ 
\p\cM^A_1=\p^{R_0}\cM^A_1 \cup \p^\infty\cM^A_1\cup \p^{int}\cM^A_1.  
$$ 
Here $\p^{R_0}\cM^A_1=\cM^A(Q_\op,Q_\up;H^{02,R_0}_s, \{f^0_p,f^2_p\},J^{02,R_0}_s,g^{02,R_0}_s)$. 
We now describe $\p^\infty\cM^A_1$, which corresponds to degenerations as $R\to\infty$. Let $p\in\cP(H_1)$, $m_0\ge 0$, $B\in H_2(W;\Z)$, and define  
$\cM^B_{m_0}(Q_\op,S_p;H^{01}_s, \{f^0_p\},J^{01}_s,g^{01}_s)$ as the union for  
$p^0_1,\dots,p^0_{m_0}\in\cP(H_0)$ and $A^0_1+\dots+A^0_{m_0}+A^{01}=B$ of the fibered products  
\begin{eqnarray*} 
&& 
W^u(Q_\op)  
\times_{\oev} 
(\cM^{A^0_1}(S_{\op}\,,S_{p^0_1};H_0,J_0,g_0)\!\times\!\R^+)\\ 
&& {_{\varphi_{f^0_{p^0_1}}\!\circ\uev}}\!\times   
_{\oev}\dots {_{\varphi_{f^0_{p^0_{m_0 -1}}}\!\circ\uev}}\!\times   
_{\oev} (\cM^{A^0_{m_0}}(S_{p^0_{m_0 -1}},S_{p^0_{m_0}};H_0,J_0,g_0)\!\times\!\R^+) \\ 
&& {_{\varphi_{f^0_{p^0_{m_0}}}\!\circ\uev}\times_{\oev}}  
\cM^{A^{01}}(S_{p^0_{m_0}},S_p;H^{01}_s,J^{01}_s,g^{01}_s). 
\end{eqnarray*} 
We define $\cM^B(Q_\op,S_p;H^{01}_s, \{f^0_p\},J^{01}_s,g^{01}_s)$ as the union over $m_0\ge 0$ of the moduli spaces $\cM^B_{m_0}(Q_\op,S_p;H^{01}_s, \{f^0_p\},J^{01}_s,g^{01}_s)$. This is a smooth manifold of dimension  
$$ 
\dim\, \cM^B(Q_\op,S_p;H^{01}_s, \{f^0_p\},J^{01}_s,g^{01}_s)=|Q_\op|-|pe^B|. 
$$ 
Given $p\in\cP(H_1)$, $m_2\ge 0$, $B\in H_2(W;\Z)$, we define  
the moduli space 
$$\cM^B_{m_2}(S_p,Q_\up;H^{12}_s, \{f^2_p\},J^{12}_s,g^{12}_s)$$ 
as the union for  
$p^2_1,\dots,p^2_{m_2}\in\cP(H_2)$ and $A^{12}+A^2_1+\dots+A^2_{m_2}=B$ of the fibered products  
\begin{eqnarray*} 
&& 
(\cM^{A^{12}}(S_p,S_{p^2_1};H^{12}_s,J^{12}_s,g^{12}_s)\!\times\!\R^+)\\ 
&& {_{\varphi_{f^2_{p^2_1}}\!\circ\uev}}\!\times   
_{\oev} (\cM^{A^2_1}(S_{p^2_1},S_{p^2_2};H_2,J_2,g_2)\!\times\!\R^+) \\ 
&&{_{\varphi_{f^2_{p^2_2}}\!\circ\uev}}\!\times   
_{\oev}\dots  
{_{\varphi_{f^2_{p^2_{m_2}}}\!\!\circ\uev}}\!\!\times 
_{\oev}   
\cM^{A^2_{m_2}}(S_{p^2_{m_2}},\!S_{\up})  
{_{\uev}\times} W^s(Q_\up). 
\end{eqnarray*} 
We define $\cM^B(S_p,Q_\up;H^{12}_s, \{f^2_p\},J^{12}_s,g^{12}_s)$ as the union over $m_2\ge 0$ of the moduli spaces $\cM^B_{m_2}(S_p,Q_\up;H^{12}_s, \{f^2_p\},J^{12}_s,g^{12}_s)$. This is a smooth manifold of dimension  
$$ 
\dim\, \cM^B(S_p,Q_\up;H^{12}_s, \{f^2_p\},J^{12}_s,g^{12}_s)=|p|-|Q_\up e^B|+1. 
$$ 
The boundary $\p^\infty\cM^A_1$ is then equal to 
$$ 
\bigcup_{\substack{p\in\cP(H_1)\\ B_0+B_2=A}}  
\hspace{-.4cm}\cM^{B_0}(Q_\op,S_p;H^{01}_s, \{f^0_p\},J^{01}_s,g^{01}_s)  
{_{\uev}\times_{\oev}} 
\cM^{B_2}(S_p,Q_\up;H^{12}_s, \{f^2_p\},J^{12}_s,g^{12}_s). 
$$ 
The boundary $\p^{int}\cM^A_1$ corresponds to degeneracies at a point $R\in]R_0,\infty[$, namely 
\begin{eqnarray*} 
\lefteqn{\p^{int}\cM^A_1}\\ 
 &\hspace{-.4cm} = & \hspace{-.4cm}\bigcup_{R> R_0} \hspace{-.1cm}\cM^{B_0}(Q_\op,Q_{p_0};H_0,\{f^0_p\},J_0,g_0) \hspace{-.1cm}\times\hspace{-.1cm} \cM^{B_2}(Q_{p_0},Q_\up;H^{02,R}_s, \{f^i_p\},J^{02,R}_s, g^{02,R}_s) \\ 
& \hspace{-.4cm}& \hspace{-.7cm}\cup\hspace{-.1cm} \bigcup_{R> R_0}\hspace{-.1cm}\cM^{B_0}(Q_\op,Q_{p_2};H^{02,R}_s, \{f^i_p\},J^{02,R}_s, g^{02,R}_s) \hspace{-.1cm}\times\hspace{-.1cm} \cM^{B_2}(Q_{p_2},Q_\up;H_2,\{f^2_p\},J_2,g_2). 
\end{eqnarray*} 
Here we used the shortcut notation $\{f^i_p\}=\{f^0_p,f^2_p\}$, and the union is taken over $B_0+B_2=A$, $p_i\in\cP(H_i)$, $Q_{p_i}\in\mathrm{Crit}(f^i_{p_i})$, $i=0,2$, such that $|Q_{p_0}|-|Q_\up e^{A-B}|=-1$ and $|Q_\op|-|Q_{p_2} e^{A-B}|=-1$. For a generic choice of the triple $(H^{02,R}_s,J^{02,R}_s,g^{02,R}_s)$, $R\ge R_0$, there are only a finite number of values of $R$ involved in the above union. The elements of $\p^{int}\cM^A_1$ correspond to the breaking of a gradient trajectory involved in one of the fiber products defining $\cM^A(Q_\op,Q_\up;H^{02,R}_s, \{f^0_p,f^2_p\},J^{02,R}_s, g^{02,R}_s)$, as $R$ converges to some $R_{int}\in]R_0,\infty[$. There are yet two other types of degeneracy in $\cM^A_1$, which compensate each other: the length of a gradient trajectory in a fibered product as above can shrink to zero, or a Floer trajectory can break at a point $Q\in S_p\setminus \mathrm{Crit}(f^i_p)$, for some $p\in\cP(H_i)$, $i=0,2$.  
We define a map $K_1:BC_*^{a,N}(H_0)\to BC_{*+1}^{a,N}(H_2)$ by 
$$ 
K_1(Q_\op)=\sum_{R>R_0} \ \sum_{|Q_\op|-|Q_\up e^A|=-1} \ \sum_{\u\in\cM^A(Q_\op,Q_\up;H^{02,R}_s, \{f^0_p, f^2_p\},J^{02,R}_s, g^{02,R}_s)} \eps(\u)Q_\up e^A. 
$$ 
The same argument as in the proof of Proposition~\ref{prop:chain-homotopy} shows that $K_1$ is of order $1$. 
The previous description of $\p\cM^A_1$ can be summarized by saying that $d\circ K_1+K_1\circ d+\sigma_{H_2,H_0}^{H^{02,R_0}_s}$  
is equal to the chain map obtained by the count of elements in $\p^\infty\cM^A_1$.  
 
We now exhibit another $1$-dimensional moduli space whose boundary contains $\p^\infty \cM^A_1$.  
Given $\op\in\cP(H_0)$, $\up\in\cP(H_2)$, 
$Q_\op\in\mathrm{Crit}(f^0_\op)$, $Q_\up\in\mathrm{Crit}(f^2_\up)$, $A\in H_2(W;\Z)$, and 
$m_0,m_1,m_2\ge 0$, we denote by   
$$ 
\cM^A_{m_0,m_1,m_2}(Q_\op,Q_\up;H^{ij}_s,\{f^k_p\},J^{ij}_s,g^{ij}_s) 
$$ 
the union for $p^0_1,\dots,p^0_{m_0}\in\cP(H_0)$, $p^1_1,\dots,p^1_{m_1+1}\in\cP(H_1)$,  
$p^2_1,\dots,p^2_{m_2}\in\cP(H_2)$, and $A^0_1+\dots+A^0_{m_0}+A^{01}+A^1_1+\dots+A^1_{m_1}+A^{12}+A^2_1+\dots+A^2_{m_2}=A$ of the fibered products  
\begin{eqnarray*} 
&& 
W^u(Q_\op)  
\times_{\oev} 
(\cM^{A^0_1}(S_{\op}\,,S_{p^0_1};H_0,J_0,g_0)\!\times\!\R^+)\\ 
&& {_{\varphi_{f^0_{p^0_1}}\!\circ\uev}}\!\times   
_{\oev}\dots {_{\varphi_{f^0_{p^0_{m_0 -1}}}\!\circ\uev}}\!\times   
_{\oev} (\cM^{A^0_{m_0}}(S_{p^0_{m_0 -1}},S_{p^0_{m_0}};H_0,J_0,g_0)\!\times\!\R^+) \\ 
&& {_{\varphi_{f^0_{p^0_{m_0}}}\!\circ\uev}\times_{\oev}}  
(\cM^{A^{01}}(S_{p^0_{m_0}},S_{p^1_1};H^{01}_s,J^{01}_s,g^{01}_s)\!\times\!\R^+)\\ 
&& {_{\varphi_{f^1_{p^1_1}}\!\circ\uev}}\!\times_{\oev} 
(\cM^{A^1_1}(S_{p^1_1}\,,S_{p^1_1};H_1,J_1,g_1)\!\times\!\R^+)\\ 
&& {_{\varphi_{f^1_{p^1_1}}\!\circ\uev}}\!\times  
_{\oev}\dots {_{\varphi_{f^1_{p^1_{m_1}}}\!\circ\uev}}\!\times   
_{\oev} (\cM^{A^1_{m_1}}(S_{p^1_{m_1}},S_{p^1_{m_1 +1}};H_1,J_1,g_1)\!\times\!\R^+) \\ 
&& {_{\varphi_{f^1_{p^1_{m_1 +1}}}\!\circ\uev}}\times_{\oev} 
(\cM^{A^{12}}(S_{p^1_{m_1 +1}},S_{p^2_1};H^{12}_s,J^{12}_s,g^{12}_s)\!\times\!\R^+)\\ 
&& {_{\varphi_{f^2_{p^2_1}}\!\circ\uev}}\!\times_{\oev}  
(\cM^{A^2_1}(S_{p^2_1},S_{p^2_2};H_2,J_2,g_2)\!\times\!\R^+) \\ 
&&{_{\varphi_{f^2_{p^2_2}}\!\circ\uev}}\!\times   
_{\oev}\dots  
{_{\varphi_{f^2_{p^2_{m_2}}}\!\!\circ\uev}}\!\!\times 
_{\oev}   
\cM^{A^2_{m_2}}(S_{p^2_{m_2}},\!S_{\up})  
{_{\uev}\times} W^s(Q_\up). 
\end{eqnarray*} 
In the above notation we abridged $H^{ij}_s=\{H^{01}_s,H^{12}_s\}$ (similarly for $J^{ij}_s,g^{ij}_s$) and $\{f^k_p\}=\{f^0_p,f^1_p,f^2_p\}$.  
We denote $\cM^A_2:=\cM^A(Q_\op,Q_\up;H^{ij}_s,\{f^k_p\},J^{ij}_s,g^{ij}_s)$ the union over $m_k\ge 0$, $k=0,1,2$ of the previously defined moduli spaces $\cM^A_{m_0,m_1,m_2}(Q_\op,Q_\up;H^{ij}_s,\{f^k_p\},J^{ij}_s,g^{ij}_s)$. This is a smooth manifold of dimension $|Q_\op|-|Q_\up e^A|+1$. In the case $|Q_\op|-|Q_\up e^A|=0$ that we are considering, $\cM^A_2$ is a smooth $1$-dimensional manifold whose boundary splits as 
$$ 
\p\cM^A_2=\p^0\cM^A_2\cup \p^{\infty,0}\cM^A_2\cup \p^{\infty,1}\cM^A_2\cup \p^{\infty,2}\cM^A_2. 
$$ 
Here $\p^0\cM^A_2$ corresponds to $m_0=0$ and the length of the gradient trajectory running between the endpoints of the two $s$-dependent Floer trajectories being equal to $0$. Thus $\p^0\cM^A_2=\p^\infty\cM^A_1$. The elements of $\p^{\infty,k}\cM^A_2$, $k=0,1,2$ correspond to the breaking of a gradient trajectory of $f^k_p$ appearing in the fibered product which defines $\cM^A_2$, for some $p\in\cP(H_k)$. Thus we have 
\begin{eqnarray*} 
\lefteqn{\p^{\infty,1}\cM^A_2}\\ 
&\hspace{-.5cm}=& \hspace{-.8cm}\bigcup_{\substack{p\in\cP(H_1) \\ Q_p\in\mathrm{Crit}(f^1_p) \\ B^{01}+B^{12}=A}} 
\hspace{-.5cm}\cM^{B^{01}}(Q_\op,Q_p;H^{01}_s,\{f^i_p\},J^{01}_s,g^{01}_s) \times \cM^{B^{12}}(Q_p,Q_\up;H^{12}_s,\{f^j_p\},J^{12}_s,g^{12}_s). 
\end{eqnarray*} 
Here we abridged $\{f^i_p\}=\{f^0_p,f^1_p\}$ and $\{f^j_p\}=\{f^1_p,f^2_p\}$. Similarly, we have  
\begin{eqnarray*} 
\lefteqn{\p^{\infty,0}\cM^A_2}\\ 
&\hspace{-.2cm}=& \hspace{-.8cm}\bigcup_{\substack{p\in\cP(H_0) \\ Q_p\in\mathrm{Crit}(f^0_p) \\ B^0+B^{02}=A}} 
\hspace{-.5cm}\cM^{B_0}(Q_\op,Q_p;H_0,\{f^0_p\},J_0,g_0) 
 \times \cM^{B^{02}}(Q_p,Q_\up;H^{ij}_s,\{f^k_p\},J^{ij}_s,g^{ij}_s) 
\end{eqnarray*} 
with $|Q_p|-|Q_\up e^{B^{02}}|=-1$, and 
\begin{eqnarray*} 
\lefteqn{\p^{\infty,1}\cM^A_2}\\ 
&\hspace{-.2cm}=& \hspace{-.8cm}\bigcup_{\substack{p\in\cP(H_2) \\ Q_p\in\mathrm{Crit}(f^2_p) \\ B^{02}+B^2=A}} 
\hspace{-.5cm}\cM^{B^{02}}(Q_\op,Q_p;H^{ij}_s,\{f^k_p\},J^{ij}_s,g^{ij}_s) \times \cM^{B^2}(Q_p,Q_\up;H_2,\{f^2_p\},J_2,g_2) 
\end{eqnarray*} 
with $|Q_\op|-|Q_p e^{B^{02}}|=-1$. We define $K_2:BC_*^{a,N}(H_0)\to BC_{*+1}^{a,N}(H_2)$ by 
$$ 
K_2(Q_\op)=\sum_{\substack{\up\in\cP(H_2) \\ B^{02}\in H_2(W;\Z) \\ |Q_\op|-|Q_\up e^{B^{02}}|=-1}} 
\  
\sum_{\u\in \cM^{B^{02}}(Q_\op,Q_\up;H^{ij}_s,\{f^k_p\},J^{ij}_s,g^{ij}_s)}  
\eps(\u) Q_\up e^{B^{02}}. 
$$ 
The same argument as in the proof of Proposition~\ref{prop:chain-homotopy} shows that $K_2$ is of order $1$.   
The previous description of $\p\cM^A_2$ shows that the chain map determined by the count of elements in $\p^0\cM^A_2$ is equal to $\sigma_{H_2,H_1}^{H_s^{12}}\circ \sigma_{H_1,H_0}^{H_s^{01}} - d\circ K_2 - K_2\circ d$. Since $\p^0\cM^A_2=\p^\infty\cM^A_1$, we obtain the conclusion of the Proposition by setting $K:=K_1+K_2$.  
\end{proof}

\begin{proof}[Proof of Proposition~\ref{prop:indepJg}] 
We consider a generic homotopy $(J^{12}_s,g^{12}_s)$, $s\in \R$ inside $\cJ^{S^1}_N$ from $(J_1,g_1)$ to $(J_2,g_2)$, which is constant near $\pm\infty$. Then $(J^{21}_s,g^{21}_s):=(J^{12}_{-s},g^{12}_{-s})$ is a homotopy from $(J_2,g_2)$ to $(J_1,g_1)$. These determine filtered chain maps $\sigma_{21}:BC_*^{a,N}(H,J_1,g_1)\to BC_*^{a,N}(H,J_2,g_2)$ and $\sigma_{12}:BC_*^{a,N}(H,J_2,g_2)\to BC_*^{a,N}(H,J_1,g_1)$.  By Proposition~\ref{prop:composition}, the composition $\sigma_{21}\circ \sigma_{12}$ is homotopic to  
the filtered chain map $\sigma_{22}:BC_*^{a,N}(H,J_2,g_2)\to BC_*^{a,N}(H,J_2,g_2)$ determined by the concatenation $(J^{21}_s\#_R J^{12}_s, g^{21}_s\#_R g^{12}_s)$ for $R>0$ large enough. The latter is homotopic to the identity by Proposition~\ref{prop:chain-homotopy}.  
 
Since all the homotopies involved are of order $1$, we obtain that $\sigma_{21}\circ \sigma_{12}$  induces on the first page $E^{a,N;1}_{*,*}(H,J_2,g_2)$ of the corresponding spectral sequence a chain morphism which is homotopic to the identity. The induced morphism on the second page $E^{a,N;2}_{*,*}(H,J_2,g_2)$ is therefore the identity. Similarly, $\sigma_{12}\circ \sigma_{21}$ induces the identity on the second page $E^{a,N;2}_{*,*}(H,J_1,g_1)$.  
 
Thus the induced morphism $\sigma_{21}:E^{a,N;2}_{*,*}(H,J_1,g_1)\to E^{a,N;1}_{*,*}(H,J_2,g_2)$ is an isomorphism. Since $\sigma_{21}$ preserves both the degree and the filtration, it follows that  
$\sigma_{21}(E^{a,N;2}_{*,1}(H,J_1,g_1))=E^{a,N;1}_{*,*}(H,J_2,g_2)$. Since $E^{a,N;2}_{*,1}(H,J_i,g_i)\simeq SH_{*+1}^{a,S^1,N}(H,J_i,g_i)$, $i=1,2$, we obtain the desired isomorphism. The fact that it does not depend on the choice of homotopy $(J^{12}_s,g^{12}_s)$ is a consequence of Proposition~\ref{prop:chain-homotopy}.  
\end{proof} 
 
\begin{remark} The isomorphism $$SH_*^{a,S^1,N}(H,J_1,g_1)\stackrel\simeq\to SH_*^{a,S^1,N}(H,J_2,g_2)$$ constructed in the proof of Proposition~\ref{prop:indepJg} is induced by the chain map  
$$SC_*^{a,S^1,N}(H,J_1,g_1)\to SC_*^{a,S^1,N}(H,J_2,g_2)$$ 
given by the count of the elements of the $0$-dimensional moduli spaces  
$$ 
\cM_{S^1}^A(S_\op,S_\up;H,J^{12}_s,g^{12}_s):=\cM^A(S_\op,S_\up;H,J^{12}_s,g^{12}_s)/S^1.  
$$ 
\end{remark}


\end{document}